\documentclass[11pt]{amsart}
\usepackage[margin=3cm]{geometry}

\usepackage[utf8]{inputenc}
\usepackage[english]{babel}
\usepackage[T1]{fontenc}
\usepackage{lmodern}
\usepackage{comment}

\usepackage{babelbib}

\usepackage{amsmath, amssymb, amsthm}   
\usepackage{amstext, amsfonts, a4, mathrsfs}
\usepackage{epsfig}
\usepackage{epic, eepic}

\usepackage[all]{xy} 
\usepackage[svgnames]{xcolor}

\usepackage{centernot}
\usepackage{array}
\usepackage{enumerate}
\usepackage{graphics,graphicx}
\usepackage{float}
\usepackage{ae,aecompl}
\usepackage{eso-pic}
\usepackage{pst-all}
\usepackage{caption}
\usepackage{subcaption}
\usepackage{pinlabel}
\usepackage{tikz}

\usepackage{stmaryrd}

\usepackage{tabularx}
\usepackage{cases}
\usepackage[shortlabels]{enumitem}

\usepackage{hyperref}

\usepackage{thm-restate}

\newtheorem{thm}{Theorem}[section]
\newtheorem{prop}[thm]{Proposition}

\newtheorem{claim}[thm]{Claim}

\newtheorem{lem}[thm]{Lemma}

\newtheorem{coro}[thm]{Corollary}

\newtheorem{thmintro}{Theorem}

\newtheorem{corintro}[thmintro]{Corollary}

\theoremstyle{definition}
\newtheorem{defi}[thm]{Definition}

\newtheorem{example}[thm]{Example}

\theoremstyle{remark}
\newtheorem{rmk}[thm]{Remark}

\numberwithin{equation}{section}

\PassOptionsToPackage{unicode}{hyperref}
\hypersetup{ 
linktocpage=true,
colorlinks=true, %
breaklinks=true, %
pdfencoding=auto,
pdftitle={}, pdfauthor={}, pdfsubject={}
}

\definecolor{mymagenta}{RGB}{255,0,255}
\definecolor{myorange}{RGB}{255,180,0}

\def\N{{\mathbb N}}    
    
\def\R{{\mathbb R}}

\def\cA{{\mathcal A}}          \def\cC{{\mathcal C}}        \def\cP{{\mathcal P}}    
   \def\cF{{\mathcal F}}

\newcommand\Aut{\operatorname{Aut}}		  %

\def\max{{\operatorname{max}}}

\newcommand{\dX}{d_{X^+}}

\author[T.~Barthelm\'e]{Thomas Barthelm\'e}
\address{Queen's University, Kingston, Ontario}
\email{thomas.barthelme@queensu.ca}
\urladdr{sites.google.com/site/thomasbarthelme}

\author[K.~Mann]{Kathryn Mann}
 \address{Institute de Math\'ematiques de Jussieu}
 \email{mann@imj-prg.fr}
\urladdr{https://webusers.imj-prg.fr/~kathryn.mann}

\author[N.~Paulet]{Neige Paulet}
\address{Queen's University, Kingston, Ontario}
\email{neige.paulet@queensu.ca}

\author[A.~Zalloum]{Abdul Zalloum}
 \address{Queen's University, Kingston, Ontario}
 \email{az32@queensu.ca}

\title[Pseudo-Anosov flows and the geometry of Anosov-like actions]{Pseudo-Anosov flows and the geometry of Anosov-like group actions} 

\begin{document}

\begin{abstract}
We show that the action on its orbit space induced by a pseudo-Anosov flow on a closed $3$-manifold (and more general Anosov-like actions) can be seen as an isometric action on a Gromov-hyperbolic space. When the flow is not $\R$-covered, we show that this action admits elements that are weakly properly discontinuous and deduce that elements of $\pi_1(M)$ that do \emph{not} represent a periodic orbit of the flow are generic for any word metric coming from a finite generating set.  We also give a number of other geometric group-theoretic results for Anosov-like group actions on bifoliated planes. 
\end{abstract}

\maketitle

\section{Introduction}

\subsection{Motivation and main theorems}
There is a rich relationship between the geometry and topology of a 3-manifold, and the dynamics of flows that manifold supports.  An important early result in this direction is Margulis' theorem \cite{anosovSmoothErgodicSystems1967} that any 3-manifold supporting an Anosov flow must have exponential growth of its fundamental group.  This was extended by Plante--Thurston to the codimension-1 Anosov case \cite{planteAnosovFlowsFundamental1972} and by Paternain for pseudo-Anosov (expansive) flows on 3-manifolds \cite{Pat93}, following Plante and Thurston's strategy.  Exponential growth can also be detected on the level of periodic orbits:  Bowen \cite{Bow72} showed that the number of periodic orbits of bounded length grows exponentially with the length, and Barthelmé--Fenley \cite{barthelmeCountingPeriodicOrbits2017} showed that the number of \emph{conjugacy classes} (thus in particular the number of elements) in $\pi_1(M)$ representing periodic orbits grows exponentially fast with the length of the shortest periodic orbit in that class.

We show here that, by contrast, relatively few loops in $M$ are represented by periodic orbits.  To make this precise, for a flow $\phi$ on a manifold $M$, denote by $\cP(\phi) \subset \pi_1(M)$ the set of elements freely homotopic to unoriented periodic orbits, and $\cP(\phi)^c$ its complement.  We prove: 

\begin{thmintro} \label{thm_generic_flows_version}
    Let $\phi$ be a non $\mathbb{R}$-covered pseudo–Anosov flow on a closed 3-manifold $M$.  Then, for any word metric on $\pi_1(M)$ coming from a finite generating set, 
    the set $\cP(\phi)^c$ is generic. 
    
    Furthermore, for any finite generating set $S$, there exists $S' \supset S$ for which the critical exponent of $\cP(\phi)^{c}$ is strictly larger than that of $\cP(\phi)$. If $\pi_1(M)$ is Gromov hyperbolic, this holds with $S'=S$.
\end{thmintro}

The $\mathbb{R}$-covered case, which consists of suspension flows of linear hyperbolic maps of the torus, and skew-Anosov flows, 
is different: in the suspension case $\cP(\phi)$ is generic (and $\cP(\phi)^c$ in fact has polynomial growth), while in the skew case we know only that the critical exponent of $\cP(\phi)^c$ is equal to that of the group, see Theorem \ref{thm: skew free maximal exponent}.  

The idea behind the proof of Theorem \ref{thm_generic_flows_version} is a novel approach to the study of pseudo-Anosov flows (and bifoliated spaces more generally) via the large scale geometry of an associated Gromov-hyperbolic graph.  We show that, for any pseudo-Anosov flow, the intersection pattern of stable and unstable leaves in the universal cover $\widetilde M$ can be captured combinatorially by an {\em intersection graph} to which the action of $\pi_1(M)$ extends. 

The natural framework for describing this is in fact not $\widetilde M$ but the {\em orbit space} of the flow, which is obtained by collapsing each (lifted) orbit in $\widetilde M$ to a point, giving a plane $P$ with two transverse, possibly singular foliations induced from $\cF^s$ and $\cF^u$, on which $\pi_1(M)$ acts by foliation-preserving homeomorphisms \cite{barbotCharacterizationAnosovFlows1995,fenleyAnosovFlows3manifolds1994} (see \cite{barthelmePseudoAnosovFlowsPlane2026} for an exposition).  

Given any bifoliated plane $P$, one can define a graph $X(P)$, whose vertices are leaves (of either foliation) with an edge if they intersect. Our first result is that this graph is a \emph{quasi-tree}:

\begin{thmintro} \label{thm_quasi_tree}
Let $(P, \cF^+, \cF^-)$ be a plane with two transverse, possibly singular foliations.  Then $X(P)$ is quasi-isometric to a tree, and the group of foliation-preserving automorphisms of $P$ acts on $X(P)$ by isometries. 
\end{thmintro}
By possibly singular in the above statement, we mean that we allow for the presence of $k$-prongs.

In the special case of orbit spaces, or more generally, in the presence of dynamics of an Anosov-like group action, there are two special cases of what the bifoliation might look like, called {\em trivial} (meaning $\mathbb{R}^2$ with the coordinate foliations) and {\em skew} (a diagonal strip, see Example \ref{ex:trivial_skew}). Otherwise, there are infinitely many different isomorphism classes of bi-foliations that may arise. 
The next theorem describes the structure and dynamics of such a group action on $X(P)$

\begin{thmintro} \label{thm_generic_plane_version}
If $G \subset \Aut(P)$ is {\em Anosov-like with an extremal Smale class} then we are in exactly one of the following cases:
\begin{enumerate}[label = (\roman*)]
    \item $X(P)$ has finite diameter, which occurs iff the plane $P$ is trivial;
    \item $X(P)$ is a quasi-line, which occurs iff the plane $P$ is skew;
    \item Otherwise, the induced action of $G$ on  $X(P)$ is nonelementary, and has elements with the {\em Weak Proper Discontinuity} property.
\end{enumerate}
\end{thmintro}
The definition of {\em Weak Proper Discontinuity} is recalled in Definition \ref{def: WPD}. The existence of such an element is the condition needed to apply many geometric group theoretic results. In this article, it leads to Theorem \ref{thm_generic_flows_version} as well as Corollary \ref{cor_acylindrically_hyperbolic} below.  
``Anosov-like" and ``extremal Smale class" are conditions on the dynamics of $G$ that capture the behavior of pseudo-Anosov flow on compact 3-manifolds, and so, in particular, are two conditions that are automatically satisfied by the orbit space action from any pseudo-Anosov flow on a closed $3$-manifold.     Smale classes are not in fact used in this work, rather we only use one elementary consequence for the structure of $P$, given in Proposition \ref{prop extremal smale}. %

As a consequence, by using work of \cite{Choi2024PseudoAnosovsGeneric,Choi2025acylindrically,GekhtmanTaylorTiozzo2022}, we conclude that the set of WPD elements for such actions is generic, and deduce Theorem \ref{thm_generic_flows_version} by translating between the orbit space and the flow.  
We also derive a number of other results of geometric group theoretic and dynamical nature, including: 

\begin{corintro} \label{cor_acylindrically_hyperbolic}
Any group with an Anosov-like action with an extremal Smale class, on a nonskew, nontrivial, bifoliated plane is acylindrically hyperbolic.  %
\end{corintro}

\begin{thmintro}[Classification of Anosov-like automorphisms] \label{thm_no_parabolics}
Let $(P, \cF^+, \cF^-)$ be a nontrivial bifoliated plane and $G$ a group with Anosov-like action. Then the action of $G$ on $X(P)$ is such that:
\begin{enumerate}[label=(\roman*)]
    \item Any loxodromic element acts freely on $P$; 
    \item Any elliptic element either fixes a point in $P$ or preserves a {\em scalloped region};
    \item There are no parabolic elements.
\end{enumerate}
\end{thmintro}
{\em Scalloped} refers to a specific feature in $P$, see Definition \ref{def_scalloped}. Note that it is a general fact (see Corollary 3.6 in \cite{MANNING2006}) that groups acting isometrically on a quasi-tree do not have parabolic elements, so the content of the previous theorem is the characterization of elliptic elements.

As was shown by Fenley \cite{fenleyIdealBoundariesPseudoAnosov2012} for flows, and \cite{bonattiActionCircleInfinity2023} generally, $P$ admits a natural compactification by a circle $\partial P$ (see Theorem \ref{thm_boundary}).   Since admitting an Anosov-like action on a nonskew, nontrivial plane forces the group to be countable \cite[Corollary 3.11]{Cam25}, we can apply work of Maher--Tiozzo \cite[Theorem 1.3]{MT18} and obtain a well-defined notion of \emph{hitting measures} on this boundary $\partial P$:

\begin{corintro}\label{cor: random walks}
    Let $(P, \cF^+, \cF^-)$ be a nonskew and nontrivial bifoliated plane and $G$ a group with Anosov-like action. Let $\mu$ be a nonelementary probability measure on $G$ with finite first moment. Then for any base point $x_0\in P$ and almost every sample path $(g_n)_{n\in \mathbb N}$ in $G$, the sequence $(g_n x_0)$ converges to a unique point in the boundary $\partial P$.
\end{corintro}
And similarly, by work of Gekhtman--Taylor--Tiozzo \cite[Theorem 1.2]{GTT18}, if $G$ is assumed to be hyperbolic, then one gets that a ``typical'' geodesics in $G$ will converge to a point in $\partial P$:
\begin{corintro}\label{cor: typical geodesic}
    Let $(P, \cF^+, \cF^-)$ be a nonskew and nontrivial bifoliated plane and $G$ a finitely generated Gromov-hyperbolic group with Anosov-like action. Let $S$ be a finite generating set for $G$ and $\nu$ the Patterson--Sullivan measure on $\partial G$ determined by $S$.
    For every $x_0\in P$ and $\nu$–almost every $\eta \in \partial G$, if $(g_n)$ is a geodesic ray in $G$ converging
to $\eta$, then the sequence $g_nx_0$ in $P$ converges to a point in the boundary $\partial P$.
\end{corintro}

In fact, there are many more results about natural measures on $\partial X$ coming from the geometric group theory literature that one can now obtain and translate to $\partial P$ thanks to Theorem \ref{thm_generic_plane_version}, see for instance \cite[Theorem 1.2]{CFFT25}. However, we will not investigate those more in the present article.

\subsection{Connections with CAT(0) cube complexes and dualizable systems} 

As some experts may already have noticed, our graph $X(P)$ is highly reminiscent of the \emph{contact graph} introduced by Hagen for $\mathrm{CAT}(0)$ cube complexes \cite{hagen2014weak}, which has become a very important tool in that theory. 
Indeed one can see the graph we consider here as the bifoliated plane counterpart of Hagen's contact graph and Theorem \ref{thm_generic_plane_version} as the counterpart of \cite[Theorem 4.1]{hagen2014weak}, the main result of that work.

The key aspect shared by these two apparently unrelated settings is a notion of \emph{walls} given by the structure of the space considered: leaves of the foliations for bifoliated planes, and hyperplanes for $\mathrm{CAT}(0)$ cube complexes. The intersection pattern of these walls then leads to $X(P)$ for bifoliated planes and the contact graph for $\mathrm{CAT}(0)$ cube complexes.

One can in fact formalize this connection by using the framework of \emph{dualizable systems} introduced by Petyt--Zalloum \cite{Petyt-Z24}: In the abstract setting of a set $S$ together with a family of walls $W$ (that can formally be defined as bipartitions of $S$), \cite{Petyt-Z24} builds a metric space from any choice of \emph{dualizable system} $\mathcal C \subset 2^{W}$. Such a $\mathcal C$, is, morally, a choice of what should be ``admissible'' sequences of walls. With the correct choice of $\mathcal C$, the quasi-treeness of the graph we build as well as \cite[Theorem 4.1]{hagen2014weak} can be seen as instances of \cite[Proposition 2.14]{PSZ2025stable}.

While we choose not to do the proof of Theorem \ref{thm_quasi_tree} using \cite[Proposition 2.14]{PSZ2025stable} in order to stay more hands-on, we explain the connection between graphs, metrics in the plane, and dualizable systems in a little more depth in Section \ref{subsec: metric on plane}.

\subsection*{Outline}

In Section \ref{sec: definition graphs and quasi-tree}, we introduce several graphs associated to general bifoliated planes, prove that they are quasi-trees and that the distinct graphs we build are all quasi-isometric, getting Theorem \ref{thm_quasi_tree}.
In Section \ref{sec_nonelementary} we start considering Anosov-like actions and construct special loxodromic elements that allow us to prove all of Theorem \ref{thm_generic_plane_version} except for the WPD property, which is dealt with in Section \ref{sec: WPD}. Corollary \ref{cor_acylindrically_hyperbolic} is also deduced at the end of Section \ref{sec: WPD}.
The proof of Theorem \ref{thm_no_parabolics} is done in Section \ref{subsec: classification isometries}.
In Section \ref{sec: genericity} we use Theorem \ref{thm_generic_plane_version} to deduce a more general version of Theorem \ref{thm_generic_flows_version} (see Theorem \ref{thm_general_genericity}), as well as Corollaries \ref{cor: random walks} and \ref{cor: typical geodesic}. We also prove genericity of loxodromic elements with a certain special type of axis (Theorem \ref{thm: genericity broken axis}) and finish the article by discussing the trivial and skew cases in Section \ref{subsec: skew and trivial}.

\subsection*{Acknowledgments}
TB was partially supported by the NSERC Discovery (RGPIN-2024-04412) and Alliance International programs (ALLRP 598447 - 24). KM was partially supported by NSF grant DMS-2505228, and part of this work was accomplished while she was in residence at SLMath, funded by NSF grant DMS-2424139. NP and AZ were partially supported by the NSERC Alliance International program (ALLRP 598447 - 24).

\section{Bifoliated planes and their associated graphs}\label{sec: definition graphs and quasi-tree}
A {\em bifoliated plane} is a plane $P$ with two topologically transverse, possibly singular foliations $\cF^+$ and $\cF^-$, whose singularities are $k$-prongs, with at most one singularity on any leaf.  A {\em $k$-prong} is a structure locally isomorphic, as a bifoliated space, to the image of a neighborhood of $0$ in the complex plane (with the horizontal and vertical or real/imaginary axis foliations), under the semi-branched cover $z \mapsto z^{k/2}$.  
A {\em face} of a leaf $l$ is the boundary of a connected component of $P \setminus l$.  Nonsingular leaves have a single face, and $k$-prongs have exactly $k$ faces.
We define three natural graphs that capture the intersection pattern of leaves.  

\begin{defi} \label{def: graph X^+}
When $(P, \cF^+, \cF^-)$ is a bifoliated plane, 
we define the graph $X^+(P)$ by 
    \begin{itemize}
    \item \textbf{Vertices:} each leaf of $\cF^{+}$ is a vertex. 
    \item \textbf{Edges:} two vertices (leaves) are adjacent whenever they are both intersected by a common non-singular leaf of $\cF^{-}$.
\end{itemize}
The graph $X^-(P)$  is defined similarly, reversing the roles of $+$ and $-$ above.  When there is no ambiguity between bifoliated planes, we will simply write $X^+$ or $X^-$.   
\end{defi}

The definition of edges can be equivalently reformulated in terms of faces, as follows: 
{\em two leaves $l, l'$ in $\cF^+$ are joined by an edge if and only if they meet a common face of a leaf of $\cF^-$.} 

In addition to $X^+$ and $X^-$, we define
\begin{defi}
    When $(P, \cF^+, \cF^-)$ is a bifoliated plane, define $X = X(P)$ to be the graph where each leaf of $\cF^+ \cup \cF^-$  is a vertex, and two vertices are connected by an edge if and only if they intersect.  
\end{defi}
Observe that any foliation-preserving homeomorphism of $P$ acts naturally on each of $X$, $X^+$ and $X^-$ by isometries.

\begin{example} \label{ex:trivial_skew}
There are two special cases of bifoliated planes which play a significant role in Anosov dynamics.  These are the {\em trivial plane}, 
isomorphic to $\mathbb{R}^2$ with the horizontal and vertical foliations, and the {\em skew plane}, which is the isomorphism class of the bi-infinite strip $\{(x,y) : x< y<x+1 \} \subset \mathbb{R}^2$
with the induced horizontal and vertical coordinate foliations. 
The reader can easily verify that in the trivial case, $X^\pm$ and $X$ are all complete graphs, i.e., of diameter 1, and in the skew case, both $X^\pm$ and $X$ are quasi-isometric to $\mathbb{R}$.  
\end{example}

In this section, we show that these objects are all coarsely the same, and all quasi-isometric to trees.   
\begin{thm} \label{thm:quasi-tree}
For any bifoliated plane $(P, \cF^+, \cF^-)$, the graphs $X^+$, $X^-$ and $X$ are all quasi-trees and are quasi-isometric to each other.
\end{thm}

Section \ref{sec_nonelementary} gives many instances when this tree is additionally nonelementary.

\subsection{Background: leaf spaces}

To analyze the geometry of $X^+$, we will frequently pass back and forth between the plane $P$, the graph $X^+$ and the {\em leaf space} of $\cF^+$.

Recall that the {\em leaf space} of a foliation is the quotient space obtained by collapsing each leaf to a point (with the topology as a quotient of $P$).  We denote the leaf space of $\cF^+$ by $\Lambda^+$. 
Whenever $\Lambda^+$ is Hausdorff, it has the structure of a real tree. However, it is typically non-Hausdorff, containing leaves which cannot be separated by open sets.  We call such leaves {\em nonseparated} -- see Figure \ref{fig:scalloped-region} for an example where a sequence of leaves limits onto an infinite set of leaves.
However, $\Lambda^+$ still has the structure of what Fenley calls a non-Hausdorff tree (see \cite{fenleyPseudoAnosovFlowsIncompressible2003})
\begin{defi}[Fenley]
A {\em non-Hausdorff tree} is a topological space $T$ satisfying the following: \begin{itemize}
    \item (1-dimensional) $T$ is a countable union of open intervals.
    \item (connected) For any $x, y$ in $T$, there is a finite chain of intervals $I_0, \ldots I_n$ with $x\in I_0$, $y \in I_n$ and $I_i \cap I_{i+1} \neq \emptyset$ for all $i$.
    \item (simply connected) If $J_1$ and $J_2$ are half-open intervals based at a point $x$ and intersecting only at $x$, and $y_i \in J_i$, then any finite chain of intervals from $y_1$ to $y_2$ (as in the second item) must contain $x$ in at least one of the intervals. 
\end{itemize}
\end{defi}
In the above definition, $T$ is permitted to be Hausdorff, in which case it is an $\mathbb{R}$-tree. In general, 
the non-Hausdorff property means that there may not be a unique shortest path between any two points; however, there is a unique ``pseudo-interval" in the following sense: 
\begin{defi}[Pseudo-interval]
For $x,y$ in a non-Hausdorff tree $T$
the {\em pseudo-interval} from $x$ to $y$, denoted $\llbracket x,y \rrbracket$ is the union of $\{x,y\}$ with the set of points $p$ in $T$ such that $x$ and $y$ lie in distinct connected components of $T \setminus \{p\}$. 
\end{defi}
Equivalently, $\llbracket x,y \rrbracket$ is the intersection of all continuous paths in $T$ from $x$ to $y$.  It 
can be decomposed into sub-intervals bounded by nonseparated leaves, as follows. 
\begin{thm}[See \cite{fenleyPseudoAnosovFlowsIncompressible2003}, section 3] \label{thm_pseudo_interval_decomp}
For any 
$x, y$ in $T$
there is a unique decomposition 
\[
\llbracket x,y\rrbracket = \bigsqcup_{i=1}^n [x_i,y_i] ,
\]
where $x_1=x$, $y_n=y$, and, for all $i$:
\begin{itemize}
    \item $[x_i,y_i]$ is an interval, and
    \item $y_i$ is non-separated from $x_{i+1}$ for each $i$
\end{itemize}
Moreover, this decomposition is minimal in the sense that any continuous path from $x$ to $y$ %
contains each interval $[x_i, y_i]$, %
and distinct $[x_i, y_i]$ have empty intersection.
\end{thm}

Note that the intervals $[x_i, y_i]$ in the decomposition are allowed to be degenerate, i.e., with $x_i = y_i$.  For example, if $x$ and $y$ are nonseparated, then the associated decomposition is given by $\llbracket x,y\rrbracket = [x,x] \cup [y,y]$.  In general one may have arbitrarily many such degenerate pairs in the decomposition of a pseudo-interval.

\subsection{$X^+$ is a quasi-tree}
To prove Theorem \ref{thm:quasi-tree}, we first 
show the following. 

\begin{prop} \label{prop:quasi-tree}
 For any bifoliated plane $(P, \cF^+, \cF^-)$, the graph $X^+$ is quasi-isometric to a tree. 
\end{prop}
To do this, we relate geodesics in $X^+$ to paths in $P$ and then apply a criterion of Manning. 

\begin{defi}[Realization in $P$ and projection to $\Lambda^+$] \label{def_projection_maps}
For a path $\gamma = v_1, e_1, v_2, e_2, \ldots v_n \subset X^+$ where $e_i$ is an edge between $v_i$ and $v_{i+1}$, 
we say that a path $\gamma_P$ in $P$ is a {\em realization of $\gamma$} if $\gamma_P$ consists of a concatenation of segments $s_1 t_1 s_2 t_2...s_n$, where $s_i \subset v_i$ and $t_i$ is a subset of a leaf of $\cF^-$ representing $e_i$.  

We define $\pi_{\Lambda^+}(\gamma)$ to be the projection of any realization of $\gamma$ to $\Lambda^+$.  
\end{defi}
In the vocabulary of \cite{fenleyIdealBoundariesPseudoAnosov2012}, $\gamma_P$ is a {\em polygonal path} in $P$. 
The fact that $\Lambda^+$ is a non-Hausdorff tree implies that $\pi_{\Lambda^+}(\gamma)$ is well-defined, as follows: 
\begin{lem}
$\pi_{\Lambda^+}(\gamma)$ is independent of the choice of realization of $\gamma$.
\end{lem}
\begin{proof}
If $t$ is a segment of a leaf of $\cF^-$ connecting leaves $v$ and $w$ in $\cF^+$, then the projection of $t$ to $\Lambda^+$ is an interval with endpoints $v$ and $w$. Since $\Lambda^+$ is a non-Hausdorff tree, there is a unique such interval. 
\end{proof}

The fact that all paths from $x$ to $y$ contain the pseudo-interval $\llbracket x, y\rrbracket $ means that $\llbracket x, y\rrbracket $ behaves similarly to a geodesic.  The next lemma more precisely relates pseudo-intervals to geodesics in $X^+$. 

\begin{lem}\label{lem:geodesics_vs_pseudo_interval}
  Let $\gamma$ be a geodesic path between $x$ and $y$ in $X^+$. Then every vertex $v\in \gamma$ is distance at most 2 in $X^+$ from 
  a point of $\llbracket x, y\rrbracket$. 
\end{lem}

\begin{proof}
Let $U$ denote the $\cF^-$-saturation of $\llbracket x, y\rrbracket$ and $N_1$ the $\cF^+$-saturation of $U$. Note that, while $U$ may not be connected, $N_1$ is because for any $u_1, u_2$ in the saturation of consecutive intervals in the decomposition of $\llbracket x, y\rrbracket$, there is some $u_1', u_2'$ in these intervals met by a common leaf of $N_1$. 
Also, $N_1$ consists of all the leaves of $\cF^+$ that are at distance (in $X^+$) at most $1$ from $\llbracket x, y\rrbracket$. 

Suppose for contradiction that $\gamma$ contains a vertex at distance $\geq 3$ from $\llbracket x, y\rrbracket$.  Let $v$ be the first such vertex along the geodesic, and let $v'$ be its predecessor so that $v'$ is distance $2$ from $\llbracket x, y\rrbracket$.  Since 
$v'$ is not in $N_1$, its image in $P$ lies in some connected component of $P\setminus N_1$ so there is a unique (by the fact that $\Lambda^+$ is simply connected) leaf $l \in \cF^+$ on the boundary of $N_1$ in $P$ separating $v$ from $U$, and this also separates $v'$ from $U$.  %
Since  
 $\pi_{\Lambda^+}(\gamma)$ is a path from $x$ to $y$, by Theorem \ref{thm_pseudo_interval_decomp} it contains $\llbracket x, y\rrbracket$.  It also contains $v$ and $v'$.     
Thus, any realization of $\gamma$ in $P$ contains at least two distinct segments of $\cF^-$, corresponding to non-consecutive distinct edges $e_i, e_j$  in the path, which cross $l$: the edge $e_i$ terminates at $v'$, which then has another edge to $v$, and later the path must return to a vertex in $\llbracket x, y\rrbracket$, with an edge $e_j$ represented by a leaf that crosses $l$ again.  
We can therefore create a strictly shorter path by replacing all vertices between $e_i$ and $e_j$ with $l$, showing that $\gamma$ was not in fact a geodesic.  
\end{proof}

The other ingredient we will use is the following criterion of Manning \cite[Theorem 4.6]{Man05}:
\begin{thm}(Manning Bottleneck Criterion) \label{lem: quasi tree iff}
    A graph is a quasi-tree if and only if there exists $K \ge 0$ such that, for every geodesic path of even length between two vertices $x,y$, with middle vertex $v$, any path $p$ from $x$ to $y$ contains a point at distance at most $K$ from $v$.
\end{thm}

Combining this with Lemma \ref{lem:geodesics_vs_pseudo_interval}, we can now quickly prove that $X^+$ is a quasi-tree. 
\begin{proof}[Proof of Proposition \ref{prop:quasi-tree}]
Let $\gamma$ be an even length geodesic between two points $x$ and $y$ in $X^+$ and $v$ its middle vertex.  Let $c$ be an arbitrary path between $x$ and $y$.  
Then $\pi_{\Lambda^+}(c)$ is a path from $x$ to $y$ in $\Lambda(\cF^+)$ so contains $\llbracket x, y\rrbracket$, as does $\pi_{\Lambda^+}(\gamma)$. By Lemma \ref{lem:geodesics_vs_pseudo_interval}, $v$ is distance at most $2$ in $X^+$ from $\llbracket x, y\rrbracket$.  Since $\pi_{\Lambda^+}(c)$ is by construction contained in the 1-neighborhood of $c$ in $X^+$, we conclude that $v$ is distance at most 3 from $c$.  
This shows that Manning's Bottleneck Criterion is satisfied with $k=3$, so $X^+$ is a quasi-tree.
The fact that foliation preserving homeomorphisms of $(P, \cF^+, \cF^-)$ act on $X^+$ by automorphisms (hence isometries) is direct from the definition.
\end{proof} 

\subsection{Quasi-isometry of $X$, $X^+$ and $X^-$}
There are natural ``inclusion'' maps $i_\pm \colon X^\pm \to X$, which sends a vertex in $X^\pm$, which is a leaf of $\cF^\pm$, to the corresponding leaf in $X$. 
To complete the proof of Theorem \ref{thm:quasi-tree}, it remains only to show:

\begin{prop}\label{prop: X+ X- QI}
    The inclusion maps $i_+ \colon X^+ \to X$ and $i_- \colon X^- \to X$ are quasi-isometries. In particular, $X^+$ and $X^-$ are quasi-isometric.
\end{prop}
Note that as a direct consequence of this proposition and Proposition \ref{prop:quasi-tree}, we deduce that $X$ is also a quasi-tree.

\begin{proof}
   To fix notations, we only do the case of $i^+$, the other one being identical. The map $i^+$ is coarsely surjective as any vertex in $X$ is at distance at most $1$ from a vertex corresponding to a leaf of $\cF^+$.
   If the foliations do not have singularities, then, by construction, two vertices $v,w \in X^+$ are at distance $1$ if and only if the corresponding vertices $i_+(v), i_+(w)$ are at distance $2$ in $X$. In the presence of singularities, we still have that if $d_{X^+}(v,w)=1$ then $d_X(i_+(v), i_+(w))=2$, but for the other direction we may have one additional case. If there exists a prong $p$ such that $\cF^-(p)$ intersects both $v$ and $w$, but no face of $\cF^-(p)$ intersects both, then the distance in $X$ is $2$ while the distance in $X^+$ is not $1$. In this case, the distance in $X^+$ is also $2$ by the following argument: Since $v\cap \cF^-(p)\neq \emptyset$, there is a nonsingular leaf intersecting both $v$ and $\cF^+(p)$, and similarly for $\cF^+(p)$ and $w$.
   In any case, we have that $d_X(i_+(v), i_+(w)) \leq 2 d_{X^+}(v,w)$, which gives the desired bound. 
   
   Finally, we will show that $d_{X^+}(v,w) \leq d_X(i_+(v), i_+(w))$. Consider $\gamma$ a geodesic in $X$ between $i_+(v)$ and $ i_+(w)$. Call $v_0,\dots, v_{k}$ the vertices in $\gamma$ that correspond to leaves of $\cF^+$, with $i_+(v) = v_0$ and $i_+(w) = v_{k}$. Let $e_1, \dots, e_{k}$ the vertices in $\gamma$, corresponding to leaves of $\cF^-$, such that $e_i$ is the unique vertex of $\gamma$ between $v_{i-1}$ and $v_i$. By definition, we have that $d_X(i_+(v), i_+(w)) = 2k$.
   If all the $e_i$ are non-singular leaves, then we can get a path $c$ in $X^+$, as a ``projection'' of $\gamma$, whose vertices are $v_0,\dots, v_{k}$ and edges correspond to $e_1, \dots, e_{k}$. If some $e_i$ is a singular leaf, then, as explained above we may have two distinct cases: either $v_i$ and $v_{i+1}$ are intersected by a common non-singular leaf, and so we can still think of $v_i, e_i, v_{i+1}$ as a path in $X^+$, or there exists a singular leaf $w_i$ and two edges $f_i, f_i'$ such that $v_i, f_i, w_i, f_i', v_{i+1}$ is a geodesic in $X^+$.
So in any case we get a path $c$ in $X^+$ between $v$ and $w$. If we call $n$ the number of indices $i$ where the singular leaf $e_i$ has to be replaced by $f_i, w_i, f_i'$, then the length of $c$ is $k + n \leq 2k$.

Therefore, $d_{X^+}(v,w) \leq = k + n \leq 2k = d_X(i_+(v), i_+(w))$ as claimed.\qedhere

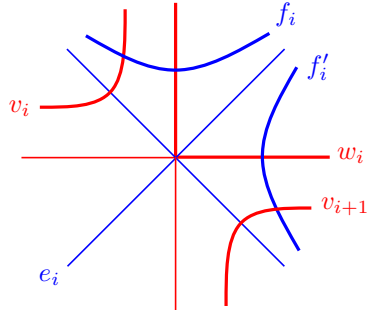
\begin{figure}[h]
    \centering
    \labellist
    \small
    \pinlabel{\color{red} $w_i$} [l] at 220 125
    \pinlabel{\color{blue} $f_i$} [bl] at 183 202
    \pinlabel{\color{blue} $f_i'$} [l] at 201 180
    \pinlabel{\color{blue} $e_i$} [tr] at 56 57
    \pinlabel{\color{red} $v_i$} [r] at 39 154
    \pinlabel{\color{red} $v_{i+1}$} [l] at 210 93
    \endlabellist
    
    \includegraphics[width=0.4\linewidth]{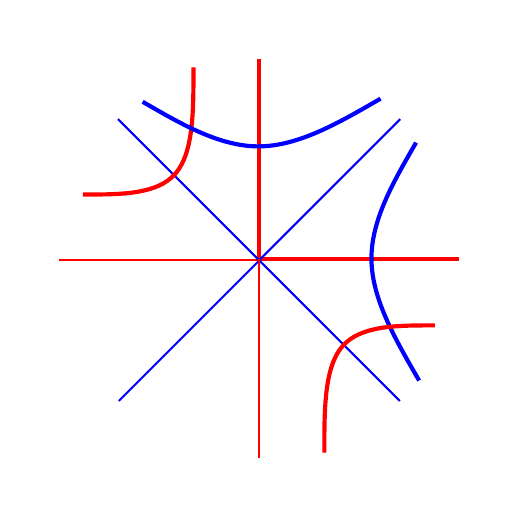}
    \caption{Case where both the distance in $X$ and $X^+$ between $v_i$ and $v_{i+1}$ is 2.}
    \label{fig: prong distance}
\end{figure}

\end{proof}

\section{Nonelementary graphs and actions} \label{sec_nonelementary}
In this section we show that, in the presence of a group acting with Anosov-like dynamics, the graph $X^+$ is typically nonelemetary and the action is nonelementary as well.  We further describe the isometries (which are necessarily elliptic or loxodromic) of $X^+$ induced by such an action. 
To do this, we begin by recalling some essential background on Anosov-like actions, more details can be found in \cite{barthelmePseudoAnosovFlowsPlane2026}.  A reader familiar with this basic theory can skip directly to Section \ref{subsec_pseudo-axes}.

\subsection{Background: Anosov-like actions on bifoliated planes}
The following definition (from \cite{barthelmePseudoAnosovFlowsPlane2026}, originally introduced in a slightly different form in \cite{barthelmeOrbitEquivalencesPseudoAnosov2022}) gives an abstraction of the key properties satisfied by the orbit-space actions of pseudo-Anosov flows on compact 3-manifolds.  
\begin{defi}[Anosov-like action] \label{def_action}
An action of a group $G$ by foliation-preserving homeomorphisms on a bifoliated plane $(P, \cF^+, \cF^-)$ is called \emph{Anosov-like} if it satisfies the following: 
\begin{enumerate}[label = (A\arabic*)]
	\item\label{Axiom_A1} 
	If a nontrivial element of $G$ fixes a leaf $l \in \cF^\pm$, then it has a fixed point $x \in l$ and is topologically expanding on one leaf through $x$ and topologically contracting on the other. 
	\item\label{Axiom_dense} The union of leaves of $\cF^{+}$ that are fixed by some element of $G$ is dense in $P$, as is the union of leaves of $\cF^{-}$ that are fixed by some element of $G$.  
	\item \label{Axiom_prongs_are_fixed} Each singular point is fixed by some nontrivial element of $G$.
	\item \label{Axiom_nonseparated} If $l$ is nonseparated with another leaf in its leaf space, then some nontrivial element $g\in G$ fixes $l$.    
\end{enumerate}
\end{defi}

\begin{thm}[Theorem 1.4.12, \cite{barthelmePseudoAnosovFlowsPlane2026}]
    If $\phi$ is a pseudo-Anosov flow on a closed $3$-manifold, the induced action of $\pi_1(M)$ on its orbit space $(\mathcal O_\phi, \cF^s, \cF^u)$ is an Anosov-like action. 
\end{thm}

Admitting an Anosov-like action constrains the topology of the foliations $\cF^+, \cF^-$.  As one important example we have: 

 \begin{thm}[Trichotomy for Anosov-like actions] \label{thm:trichotomy}
 Let $(P,\cF^+, \cF^-)$ be a bifoliated plane with an Anosov-like action of a group $G$. Then exactly one of the following holds:
  \begin{enumerate}[label=(\roman*)]
   \item  $(P,\cF^+, \cF^-)$ is trivial,
   \item  $(P,\cF^+, \cF^-)$ is skew (as defined in Example \ref{ex:trivial_skew}), or
   \item There is either a singular point in $P$, or the leaf spaces of $\cF^+$ and $\cF^-$ are both non-Hausdorff with two-sided branching. \footnote{Two-sided branching is a definition which applies to transversely orientable foliations. This is always the case when $\cF^\pm$ have no prongs.} %
  \end{enumerate}
 \end{thm}
This was proved for flows by a combination of results of Barbot \cite[Théorème 4.1]{barbotCharacterizationAnosovFlows1995} and Fenley \cite[Theorem 3.4]{fenleyAnosovFlows3manifolds1994} and \cite[Main Theorem]{fenleySidedBranchingAnosov1995}, see \cite{barthelmePseudoAnosovFlowsPlane2026}
for a proof for Anosov-like actions.  

We recall some important vocabulary describing features of bifoliated planes.  Given any set $A\subset P$, the \emph{saturation} of $A$ by $\cF^+$-leaves is the union of all $\cF^+$-leaves intersecting $A$, and we denote it by $\cF^+(A)$. Similarly, $\cF^-(A)$ denote the saturation of $A$ by leaves of $\cF^-$. Note that we often consider $\cF^+(A)$ both as a subset of $P$ and as a subset of the leaf space $\Lambda^+$.

\begin{defi}[Perfect fits and lozenges]
Two leaves $l^+, l^-$ (or rays of leaves $r^+$, $r^-$) in $\cF^+, \cF^-$, respectively, form a {\em perfect fit} if they do not intersect, but $l^\pm$ lies in the closure of the saturation $\cF^\pm(l^\mp)$. 
A {\em lozenge} $L$ is an open region, homeomorphic via a foliation-preserving homeomorphism, to the product-foliated rectangle $(0,1)^2$, and bounded by leaves $l_1^\pm, l_2^\pm$ where the pairs $l_i^\pm$ make a perfect fit, and $l_i^+ \cap l_j^- \neq \emptyset$ for $i \neq j$.   These intersection points are called the {\em corners} of the lozenge, and the rays of leaves $l_i^\pm$ in the boundary of $L$ are called the {\em sides}.  The closure of $L$ in $P$ is called a {\em closed lozenge}.  
\end{defi}

\begin{defi}[Chains]
A {\em chain of lozenges} $\cC$ is a union of closed lozenges
that satisfies the following connectedness property: for any two lozenges $L, L'$ in $\cC$,
there exist lozenges $L_0,L_1, \ldots L_k$ in $\cC$ such that $L= L_0$, $L'= L_k$ and for all $i$ the
pair of lozenges $L_i$ and $L_{i+1}$ share a corner (possibly a side).
\end{defi}
We say that a chain of lozenges is \emph{maximal} if it is maximal for the inclusion.

\begin{figure}[h]
    \centering
    \includegraphics[width=1\linewidth]{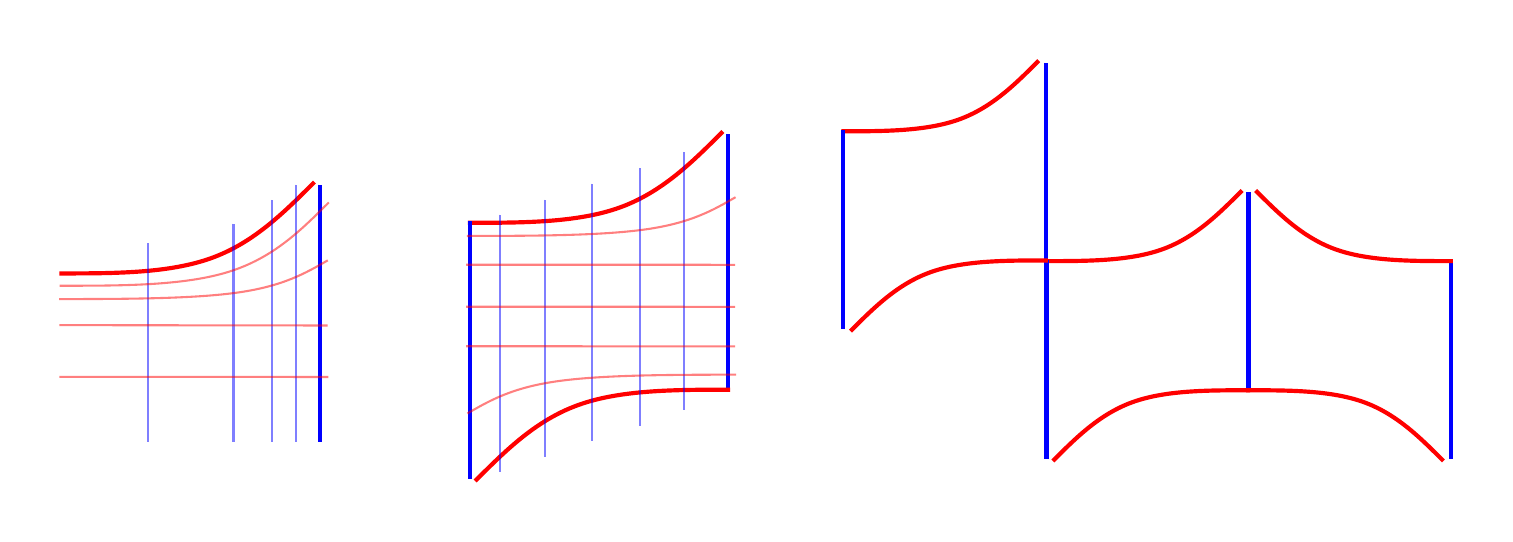}
    \caption{A perfect fit, a lozenge, and a chain of 3 lozenges.}
    \label{fig: lozenge}
\end{figure}

\begin{defi}[(non)-corners]
A point $p \in P$ is called a {\em corner} if it is a corner of some lozenge, and a {\em non-corner} otherwise.  
 A {\em non-corner fixed point} is a point $p$ which is a noncorner, and fixed by some nontrivial $g \in G$.
\end{defi}
As an example, in the skew plane, each point $p$ is the corner of exactly two lozenges, which lie in diagonal quadrants of $p$.

We will not formally define {\em Smale classes} or their partial ordering here (see \cite[Definition 3.4]{barthelmeNONTRANSITIVEPSEUDOANOSOVFLOWS}), these are generalizations of the invariant {\em basic sets} for an Anosov flow, with the Smale order.  
Transitive actions have a unique Smale class.  We assume here that the actions we consider have an {\em extremal} Smale classs.  This corresponds in the flows setting to either being transitive, or having at least one attractor or one repeller in the nontransitive case.  Thus this property always holds for the orbit space actions induced from pseudo-Anosov flows on closed 3-manifolds.   

We use this only to obtain the following consequence:
\begin{prop}[Consequence of extremal Smale class] \label{prop extremal smale}
    If an Anosov-like action of $G$ on $P$ has an extremal Smale class, and $P$ is not skew nor trivial, then there exist non-corner fixed points $a$, $b$ such that $\cF^+(a) \cap \cF^-(b) \neq \emptyset$ and $\cF^-(a) \cap \cF^+(b) = \emptyset$. 
\end{prop}
As noted above, this is automatically true for orbit space actions coming from pseudo-Anosov flows, but may fail for some exotic Anosov-like actions (see Example 8.4 and Proposition 8.10 in \cite{barthelmeNONTRANSITIVEPSEUDOANOSOVFLOWS}).

\begin{proof}
The proof for transitive actions is given in \cite[Lemma 4.11]{barthelmeOrbitEquivalencesMathbb2023}, using the fact that non-corner points are dense in $P$.  For an extremal Smale class $\Lambda$, \cite[Proposition 6.8]{barthelmeNONTRANSITIVEPSEUDOANOSOVFLOWS} shows that non-corner fixed points are dense in $\Lambda$, and, by definition of extremal, $\Lambda$ is saturated by either $\cF^+$ or $\cF^-$ \cite[Observation 6.6]{barthelmeNONTRANSITIVEPSEUDOANOSOVFLOWS}. With these properties the proof of \cite[Lemma 4.11]{barthelmeOrbitEquivalencesMathbb2023} carries over unchanged.  
\end{proof}

The relationship between chains of lozenges and dynamics of the action comes from the following theorem, originally proved by Fenley in the context of Anosov flows:
\begin{thm}[See Theorem 2.8.7 of \cite{barthelmePseudoAnosovFlowsPlane2026}] \label{thm_fixed_connected}
Suppose $G$ acts Anosov-like on $P$.  If a nontrivial element $g \in G$ fixes distinct points $x, y \in P$, then $x$ and $y$ are corners of a chain of lozenges.  Conversely, if a nontrivial $g$ fixes some corner $x$ of a chain $\mathcal{C}$ of lozenges, then (up to passing to a power $g^k$ preserving all quadrants of $x$) $g^k$ fixes all corners of $\mathcal{C}$.  
\end{thm}
Related to this, we also have
\begin{prop}[See \cite{barthelmePseudoAnosovFlowsPlane2026}, Proposition 2.5.3 and Section 2.5] \label{prop:nonseparated_on_chain}
For any collection of pairwise nonseparated leaves in $P$, there is a chain of lozenges, fixed by some nontrivial $h \in G$, and with at least one lozenge having a side on each leaf.   
\end{prop}

There is one isomorphism class of chain that plays a special role, called a {\em scalloped region}, it will appear a few times in our arguments, (for example in the proof of Lemma \ref{lem: connect fix sets}). 

\begin{defi} \label{def_scalloped}
 A scalloped region is an open region $U \subset P$ that can be realized as the interior of a chain of lozenges $\{L_i : i \in \mathbb{Z}\}$ such that for all $i$, $L_i$ and $L_{i+1}$ have a $\cF^+$-side in common, {\em and} can also be realized as the interior of a chain $\{L'_i : i \in \mathbb{Z}\}$ such that for all $i$, $L'_i$ and $L'_{i+1}$ have a $\cF^-$-side in common.   See Figure \ref{fig:scalloped-region}.
\end{defi}

\begin{figure}[h]
    \centering
\includegraphics[width=0.75\linewidth]{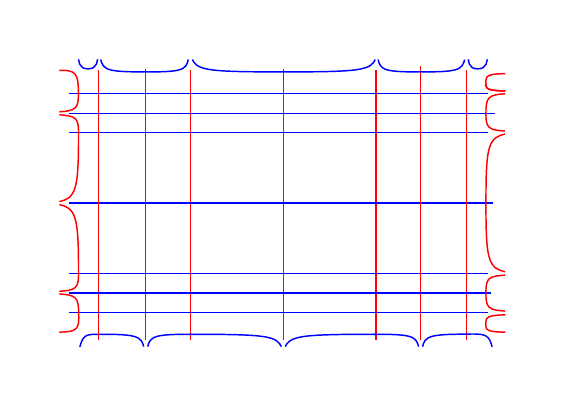}
    \caption{A scalloped region}
    \label{fig:scalloped-region}
\end{figure}

\subsection*{Compactification}
Fenley \cite{fenleyIdealBoundariesPseudoAnosov2012} constructed a compactification of the orbit space of an Anosov flow by a ``circle at infinity"; a variant on this construction can also be used to compactify any bifoliated plane (see \cite{bonattiActionCircleInfinity2023}). 
 We will primarily use this as a way to keep track of the cyclic order of ends of leaves, and to constrain the dynamics of individual elements.  
We summarize the relevant properties below, taken from 
Fenley's work \cite{fenleyIdealBoundariesPseudoAnosov2012} and
Bonatti's generalization \cite{bonattiActionCircleInfinity2023}.  See \cite[Section 3.1]{barthelmePseudoAnosovFlowsPlane2026} for an exposition.  

\begin{thm} \label{thm_boundary}
Let $(P, \cF^+, \cF^-)$ be a bifoliated plane.  Then $P$ admits a compactification by a circle $\partial P$ such that $P \cup \partial P$ is a closed disc, 
and the following hold: 
\begin{enumerate}[label=(\arabic*)]
\item Each ray $r$ of each leaf limits to a unique point of $\partial P$, called the {\em endpoint of $r$}.
\item\label{item:same endpoint subset of Z} The set of rays limiting to the same boundary point is countable, and can be indexed by a (finite or infinite) subset $I \subset \mathbb{Z}$ such that $r_i$ makes a perfect fit with $r_{i+1}$ for all $i \in I$.  
\item \label{item_perfect_fit} Any two rays making a perfect fit have the same limit. 
\item The set of rays limiting to any open subset of $\partial P$ is uncountable.  
\end{enumerate}
Moreover, there is a unique compactification with the three properties above, and any group action on $P$ by automorphisms extends to act on $\partial P$ by homeomorphisms. 
\end{thm}
The second part of Item \ref{item:same endpoint subset of Z} and Item \ref{item_perfect_fit} are not stated in \cite{bonattiActionCircleInfinity2023}, but Item \ref{item_perfect_fit} follows directly from the construction.  
Item  \ref{item:same endpoint subset of Z} is explained in \cite[Section 3.1]{barthelmePseudoAnosovFlowsPlane2026} under the assumption that $P$ admits an Anosov-like action (which is all we consider here). The general statement can be proved by first showing that if two leaves $l_1,l_2$ of $\cF^+$ have a shared endpoint, then there exists at least one leaf $l^-\in \cF^-$ that separates $l_1$ from $l_2$, which follows easily from the fact that at most countably many leaves share the same endpoint.%

The action of specific elements on $\partial P$ can also easily be described: 

 \begin{prop}[Boundary action, see \cite{barthelmePseudoAnosovFlowsPlane2026}, Proposition 3.1.14] \label{prop:boundary_action}
Let $(P,\cF^+, \cF^-)$ be a nontrivial bifoliated plane with Anosov-like action of $G$, and $g \in G$ nontrivial.  
\begin{itemize} 
\item If $g$ fixes a non-corner point $x$, then the only points of $\partial P$ fixed by $g$ are endpoints of $\cF^\pm(x)$.

\item If $g$ fixes all corners in a maximal chain of lozenges $\mathcal{C}$, then the set of fixed points of $g$ on $\partial P$ is the closure of the set of endpoints of leaves of the sides of lozenges in $\mathcal{C}$.

\item If $g$ acts freely on $P$, then it either has one or two fixed points on $\partial P$, or has exactly four fixed points and preserves a scalloped region.  
\end{itemize} 
\end{prop} 

\begin{rmk} 
    Note that the three possibilities listed in the above proposition are not exhaustive: One may have an element $g\in G$ that fixes a unique point $p$ but such that $p$ is the corner of a lozenge. However, in this case, there exists $k$ such that $g^k$ fixes all corners in the maximal chain of lozenges containing $p$. In other words, the list above becomes exhaustive only up to replacing $g$ by a power when necessary.  
\end{rmk}

\subsection{Elements with distinct quasi-axes in $X^+$} \label{subsec_pseudo-axes}
Going forward, we assume $G$ acts on $P$ with an Anosov-like action admitting an extremal Smale class. The first step to show $X^+$ is nonelementary is to find elements in $G$ which act loxodromically on $X^+$.  To do this, we first build elements with specific axes in $\Lambda^+$, then relate these to 
quasi-axes in $X^+$. 

Since $\Lambda^+$ is a non-Hausdorff tree,  any homeomorphism acting freely has an invariant {\em axis}, defined as the set of points $x \in \Lambda^+$ such that $x$ separates $g^{-1}(x)$ from $g(x)$. See \cite[Theorem A]{fenleyPseudoAnosovFlowsIncompressible2003}.  Such axes can be thought of as bi-infinite pseudo-intervals, and have a similar structure 
as pseudo-intervals: an axis is either homeomorphic to $\mathbb{R}$, or can be uniquely decomposed into a union of intervals $[x_i, y_i]$, indexed by $\mathbb{Z}$ such that $x_i$ is nonseparated with $y_{i-1}$.

\begin{prop}
  \label{prop: free element condition}
    Suppose $P$ is not skew nor trivial and $G$ has a Smale-bounded Anosov-like action on $P$. Then, there exists an element $g \in G$ acting freely on $P$ such that:
    \begin{enumerate}
        \item The axis $\cA(g)$ of $g$ in the leaf space $\Lambda^+$ has one of the following forms:
        \begin{enumerate}
            \item \label{it: non sep}
            $\cA(g) = \bigcup_i [x_i,y_i]$ with $y_i$ non-separated from $x_{i+1}$; or
            \item \label{it: prong}
            $\cA(g)$ is a properly embedded copy of $\mathbb R$ containing a prong leaf $l$ with singularity $p$ dividing $\cA(g)$ (See Definition \ref{def: separating prong} below).
        \end{enumerate}
        \item \label{it: not chain}
            The axis $\cA(g)$ is not contained inside the projection to $\Lambda^+$ of any chain of lozenges.
    \end{enumerate}
\end{prop}

\begin{rmk}
    Theorem \ref{thm: genericity broken axis} gives quite general conditions under which elements as in case \ref{it: non sep} not only exist, but are generic.
\end{rmk}

\begin{defi}[Dividing prong singularity] \label{def: separating prong} 

Let  $[x,y]$ be an interval in $\Lambda^+$, and $p$ a prong singularity with $\cF^+(p) \in (x,y)$.  We say $\cF^+(p)$ (or just $p$ for short) {\em divides the interval $[x,y]$} if the leaf $x$ is contained in a single quadrant $Q_1$ of $p$, and 
 denoting by $Q_1, \dots, Q_{2n}$ the quadrants of $p$, enumerated in cyclic order, the leaf $y$ does not intersect 
 $Q_{2n-1}, Q_{2n}, Q_1, Q_2$ nor $Q_3$.
 
 We say a prong $p$ divides a properly embedded copy of $\R$ if it divides some compact subinterval. 
\end{defi}

\begin{figure}[h]
\centering
\labellist
\small
\pinlabel{$p$} [bl] at 117 123
\pinlabel{$p$} [bl] at 352 123

\pinlabel{\color{red} $x$} [tr] at 39 73
\pinlabel{\color{red} $y$} [tl] at 200 76

\pinlabel{\color{myorange} $\R \subset \Lambda(\cF^+)$} [r] at 26 121

\endlabellist
\includegraphics[width=.75\textwidth]{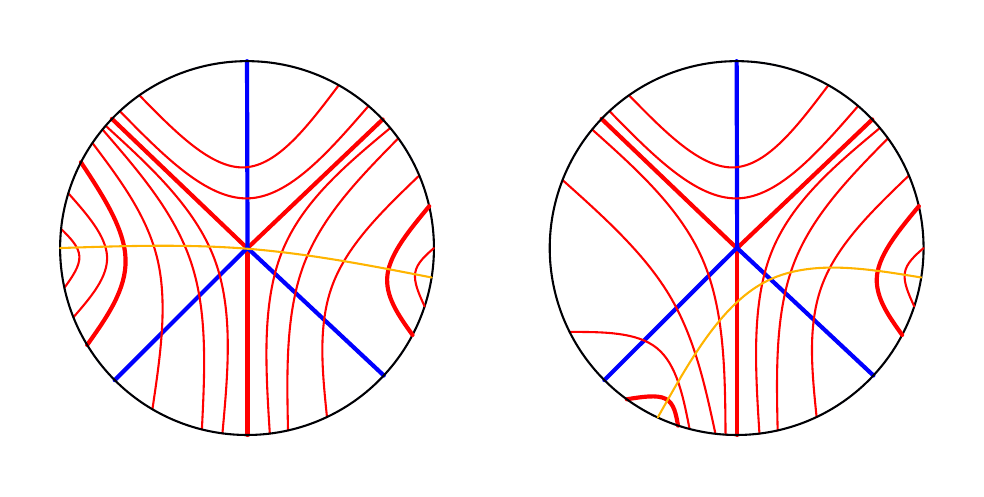}
\caption{Dividing and non dividing prongs of a copy of $\R$ in $\Lambda(\cF^+)$}
\end{figure}
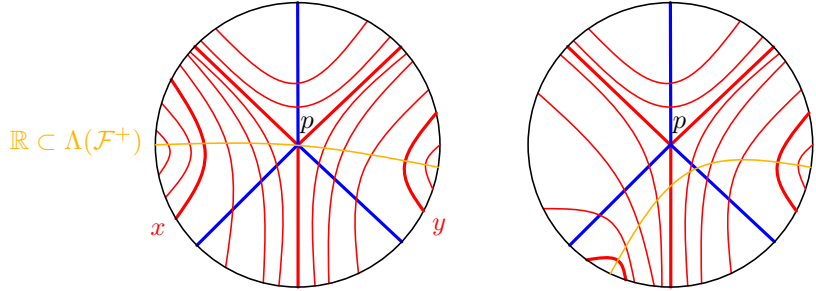

\begin{rmk}  \label{rmk: separating prong separates}
When $p$ divides $[x,y]$, it in fact forms a kind of ``division" between the stable and/or unstable saturations of the quadrants containing $x$ and $y$: if $x$ is in $Q_1$, then any leaf of $\cF^\pm$ intersecting $Q_1$ is disjoint from any leaf intersecting $y$. 
In particular, $x$ and $y$ are distance at least $2$ in $X^+$.

We will often use the terminology $p$ divides $x$ from $y$ to mean that $p$ divides $[x,y]$.  
\end{rmk}

To prove Proposition \ref{prop: free element condition}, we need two preliminary technical lemmas.
Here and throughout, when $S$ is a connected topological space (typically $P$, $\partial P$, or $P \cup \partial P$), we say that a subset $E \subset S$ {\em separates} subsets $A$ and $B$ in $S$ if $A$ and $B$ lie in distinct connected components of $S \setminus E$.  

\begin{lem} 
\label{lem: connect fix sets}
Let $g \in G$. Suppose that there exists two open disjoint sets $U,V$, such that, for all $n\neq 0$, the fixed set of $g^n$ on $\partial P$ is contained in $U\cup V$, and both $U$ and $V$ contains at least one fixed point of $g^n$. If there exists $x \in P$ such that the set of endpoints on $\partial P$ of $\cF^+(x)$ and of $\cF^-(x)$ both separate $U$ and $V$, then $g$ acts freely and fixes exactly 2 points on $\partial P$.
\end{lem} 

\begin{proof} 
Suppose for contradiction that $g$ fixes a point in $P$. Up to considering a power $g^n$, we have that either $g^n$ fixes a unique non-corner point $p$ in $P$ or it fixes all the corners of a maximal chain of lozenges. 
By Proposition \ref{prop:boundary_action}, we have that in the first case the only fixed points of $g^n$ on $\partial P$ are the endpoints of $\cF^{\pm}(p)$, and in the later case, the set of fixed points of $g^n$ in $\partial P$ is the closure of the set of endpoints of the leaves supporting the sides of lozenges in the maximal $g^n$-invariant chain. (Note that by Theorem \ref{thm_fixed_connected} there is a unique maximal $g$-invariant chain of lozenge.) However, connectedness of chains of lozenges implies that, under the hypothesis of the lemma, some leaf fixed by $g$ will have to join $U$ and $V$, hence would need to cross both $\cF^+(x)$ and $\cF^-(x)$, a contradiction. 

Therefore $g$ must act freely.  We now use the classification of induced actions on the boundary given by Proposition \ref{prop:boundary_action}.
By our assumption, $g$ fixes at least $2$ points.  In the case of preserving a scalloped region $R$, since $g$ acts freely, and preserves the decomposition of the region into lozenges, it must act as a translation along each of the ``lines" of adjacent lozenges forming the scalloped region.  Thus, every point $y\in R$ will be such that $\cF^+(y) \cup \cF^-(y)$ separates each of the 4 fixed boundary points from each other, while for any point $y\notin R$ either $\cF^+(y)$ or $\cF^-(y)$ does not separate any of the fixed boundary points, so we get a contradiction with the assumption.
Therefore, $g$ fixes exactly 2 points on $\partial P$ as claimed. 
\end{proof}

We will now start constructing the element $g$ of Proposition \ref{prop: free element condition}, starting with one lemma building free elements that will be used to get $g$. The proofs are easiest to understand by just looking at the figures, which summarize the configurations of the relevant objects.  However, to formally describe theses, we need to introduce some notation and definitions.

Following \cite[Definition 4.1]{barthelmeOrbitEquivalencesPseudoAnosov2022}, we say that two points $a,b\in P$ are \emph{partially linked} if exactly one of the two possible intersections $\cF^+(a)\cap \cF^-(b)$ and $\cF^-(a)\cap \cF^+(b)$ is non-empty.

For an element $a\in P$, denote by $\partial \cF^+(a)$, resp.~$\partial \cF^-(a)$, the set of endpoints of the $\cF^+$, resp.~$\cF^-$, leaf through $a$, and we write $\partial \cF(a) = \partial\cF^+(a)\cup\partial \cF^-(a)$.
For two points $a,b\in P$ such that $\partial \cF(a) \cap \partial \cF(b) = \emptyset$ (which happens for instance whenever the two points $a,b$ are non-corners), observe that $a,b$ are partially linked if and only if $\partial \cF(a)$ is contained in exactly two connected components of $\partial P \smallsetminus \partial \cF(b)$.

\begin{lem} \label{lem: non corner partially linked}
Suppose $\alpha, \beta \in G$ fix points $a, b$ which are partially linked and not corners and $\cF^+(a)\cap \cF^-(b) \neq \emptyset$. Let $a_1^-,a^-_2\in \partial\cF^-(a)$ be such that $\partial \cF(b)$ is contained in the interval $(a_1^-,a_2^-)$ for the counterclockwise orientation of $\partial P$. Call $a^+$ the only element in $\partial \cF^+(a) \cap (a_1^-,a_2^-)$ and let $b^-\in \partial\cF^-(b)$ and $b^+\in \partial \cF^+(b)$ be such that $(b^-, b^+)$ is the connected component of $\partial P \smallsetminus \partial \cF(b)$ that contains $a^+$.
Finally, let $U_1= (a_1^-,\eta_1)$, $V_1 = (b^-,\xi_1)$, $U_2 = (\eta_2, b^+)$ and $V_2 = (\xi_2, a_2^-)$  be any pairwise disjoint open intervals in $\partial P$. See Figure \ref{fig: partially linked a b}.

Then, for any sufficiently large $n$, the element $\alpha^n \beta^n$ acts freely, with an attractor in $U_1$ and a repeller in $U_2$
and $\alpha^n \beta^{-n}$ acts freely, with attractor in $V_1$ and repeller in $V_2$. 
\end{lem}

\begin{figure}[h]
\centering
\labellist
  \small
  \pinlabel {\color{blue} $a_2^-$} [tr] at 45 215
  \pinlabel {\color{blue} $a_1^-$} [tr] at 68 57
  \pinlabel {\color{red} $a^+$} [tl] at 215 58

  \pinlabel {\color{blue} $b^-$} [tr] at 173 32
  \pinlabel {\color{red} $b^+$} [bl] at 251 102
  \pinlabel {\color{red} $b_2^+$} [br] at 167 256
  
  \pinlabel {$a$} [tl] at 82 100
  \pinlabel {$b$} [tl] at 203 141
  \pinlabel {$p$} [br] at 168 88

  \pinlabel {\color{green} $V_1$} [bl] at 170 40
  \pinlabel {\color{green} $V_2$} [tl] at 56 213
  
  \pinlabel {\color{myorange} $U_1$} [bl] at 76 51
  \pinlabel {\color{myorange} $U_2$} [br] at 244 96
  
\endlabellist
\includegraphics[width=.5\textwidth]{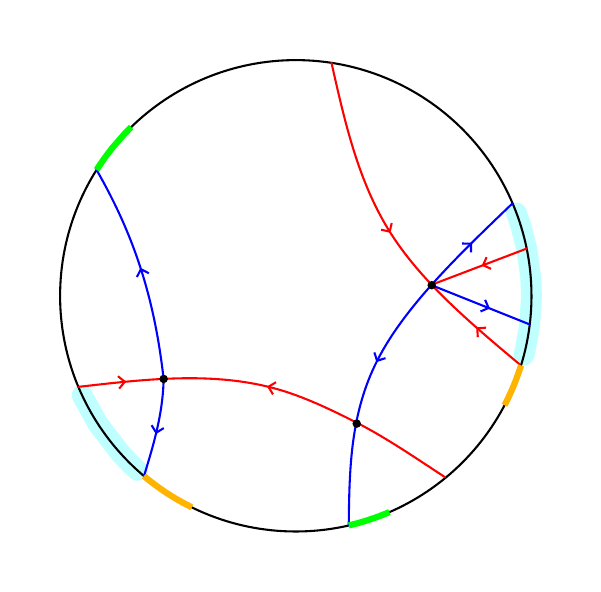}
    \caption{Partially linked non-corner fixed points $a$ and $b$}
    \label{fig: partially linked a b}
\end{figure}
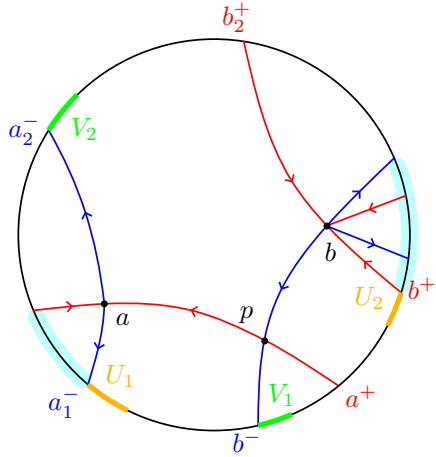

In the figure, $a$ and $b$ may be singular or not, the important configuration specified by the technical statement in the lemma is that any additional rays of their leaves lie in the shaded blue regions.   

\begin{proof}
Let $b_2^+\in \partial \cF^+(b)$ be the point such that the interval $(b_2^+, b^-)$ is the connected component of $\partial P \smallsetminus \partial \cF(b)$ containing $a_1^-,a_2^-$.
Consider $I, J$ two compact intervals in $\partial P$ such that $(a_2^-,b^-)\subset I \subset (b_2^+,a^+)$, $(b^-,b_2^+) \subset J\subset (\eta_1^-,a_2^-)$ and $I\cup J = \partial P$, as in Figure \ref{fig: intervals I J}; note that $U_1 \subset I$ and is disjoint from $J$, and $U_2 \subset J$ and is disjoint from $I$.

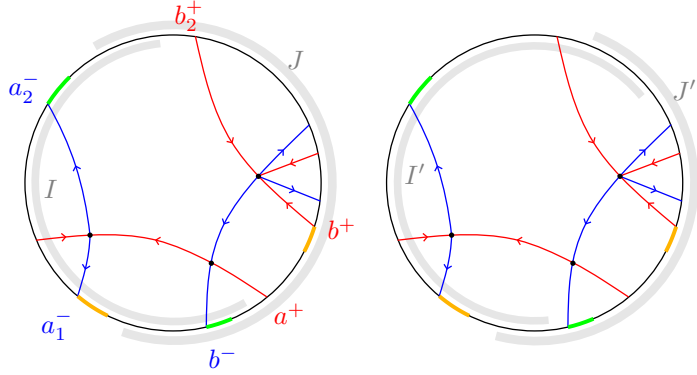
\begin{figure}[h]
\centering
\labellist
  \small
  \pinlabel {\color{gray} $I$} [r] at 60 148
  \pinlabel {\color{gray} $J$} [bl] at 225 236

    \pinlabel {\color{gray} $I'$} [br] at 340 148
  \pinlabel {\color{gray} $J'$} [tl] at 520 236

    \pinlabel {\color{blue} $a_2^-$} [br] at 45 210
  \pinlabel {\color{blue} $a_1^-$} [tr] at 70 60
  \pinlabel {\color{red} $a^+$} [tl] at 215 65

  \pinlabel {\color{blue} $b^-$} [tl] at 165 32
  \pinlabel {\color{red} $b^+$} [tl] at 256 130
  \pinlabel {\color{red} $b_2^+$} [br] at 170 265

\endlabellist
    \includegraphics[width=0.75\linewidth]{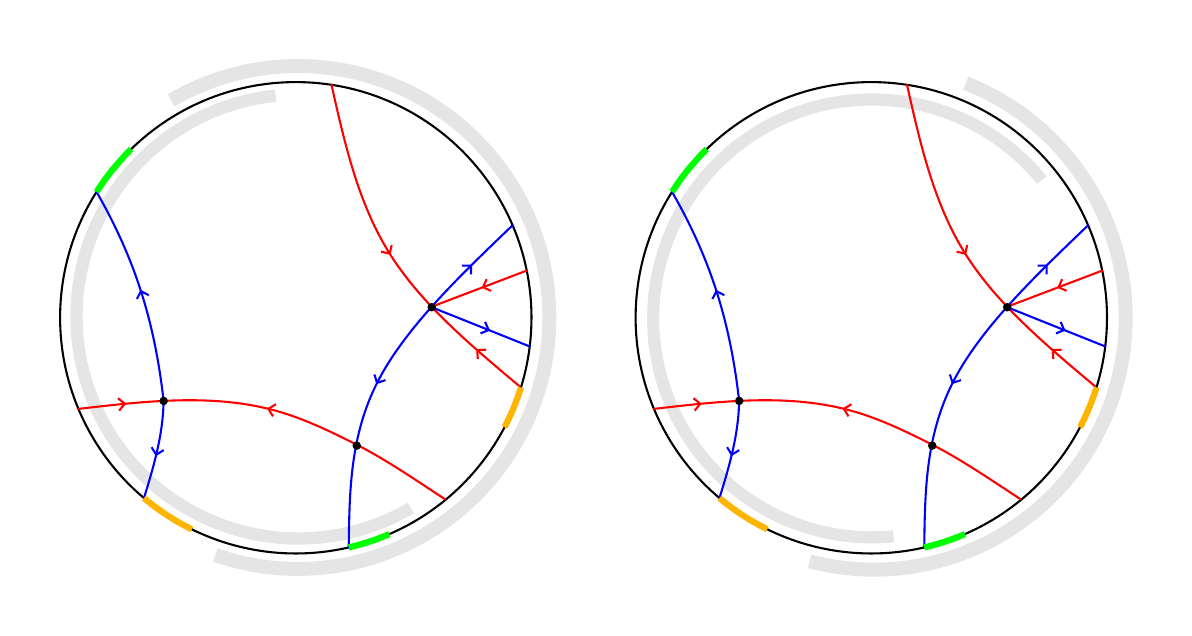}
    \caption{Intervals $I$ and $J$}
    \label{fig: intervals I J}
\end{figure}

For $n$ sufficiently large, $\alpha^n \beta^n(I) \subset U_1$ and $(\alpha^n\beta^n)^{-1}(J) \subset U_2$, so $\alpha^n \beta^n$ must admit at least one fixed point in $U_1$ and one in $U_2$. Moreover, since $I \cup J = S^1$, this means that the only fixed points of $\alpha^n \beta^n$, and all its powers, lie in $U_1 \cup U_2$.  
By Lemma \ref{lem: connect fix sets}, since each leaf of $\cF^\pm(p)$ separates $U_1$ from $U_2$, the element $\alpha^n \beta^n$ acts freely. By the above description $U_1$ contains the attractor and $U_2$ contains the repeller.  

Considering intervals $I', J'$ containing different endpoints of leaves of $a, b$ as shown in Figure \ref{fig: intervals I J}, the same argument proves the claim about $\alpha^n \beta^{-n}$.
\end{proof}

We now prove Proposition \ref{prop: free element condition}

\begin{proof}[Proof of Proposition \ref{prop: free element condition}]
Let $a$, $b$ be noncorner points as in Lemma \ref{lem: non corner partially linked}, these exist by Proposition \ref{prop extremal smale}.  
Consider the saturation $S = \cF^+(\cF^-(a))$. It contains $\cF^+(a)$ and does not contain $\cF^+(b)$. Thus there exists a unique leaf $l$ in $\partial S$ that separates $\cF^+(a)$ from $\cF^+(b)$. 
There can be two distinct cases (see Figure \ref{fig: broken and prongs}):
\begin{enumerate}[label=(\arabic*)]
\item \label{item nonseparated case}  $l$ is part of a family of non-separated leaves $\{l_i\} \subset \partial S$ with $l_1$ intersecting $\cF^-(a)$ and $l=l_n$ intersecting $\cF^-(b)$, or 
\item \label{item prong case} $l$ is a singular leaf intersecting both $\cF^-(a)$ and $\cF^-(b)$.
\end{enumerate}

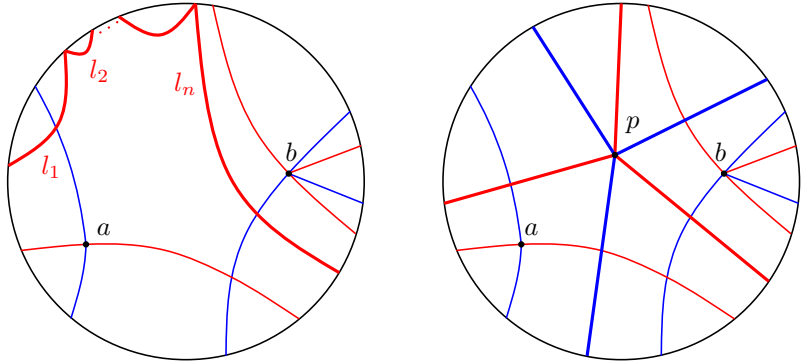
\begin{figure}[h]

    \labellist
  \small

\pinlabel{\color{red} $l_1$} [tl] at 44 142
\pinlabel{\color{red} $l_2$} [tl] at 70 194
\pinlabel{\color{red} $l_n$} [tr] at 134 187

\pinlabel{$a$} [bl] at 74 93
\pinlabel{$b$} [b] at 182 132
\pinlabel{$p$} [tl] at 360 163

\pinlabel{$a$} [bl] at 305 93
\pinlabel{$b$} [b] at 414 132

\endlabellist
\includegraphics[width=0.9\linewidth]{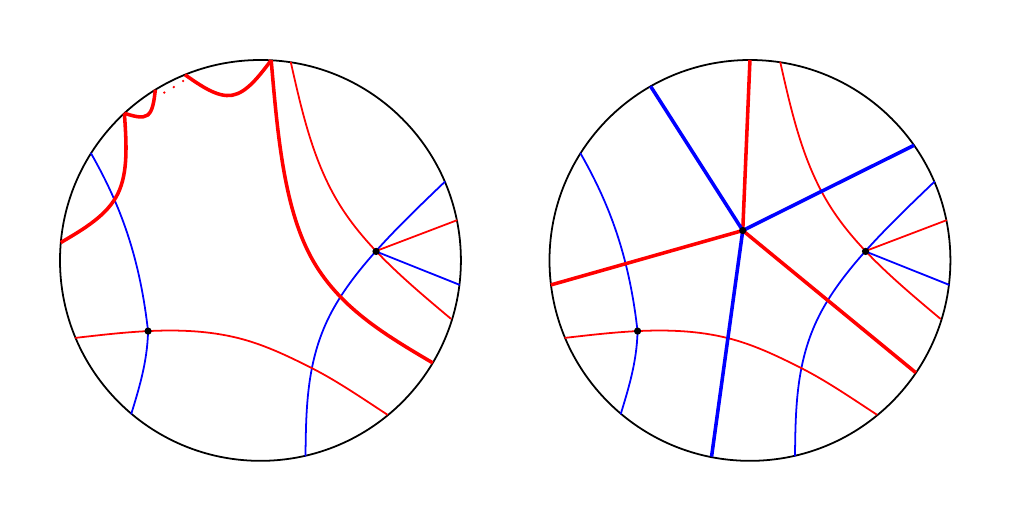}
\caption{Two configurations for $a$ and $b$}
    \label{fig: broken and prongs}
\end{figure}

Now we will consider an element $g_{m,k} = (\alpha^m \beta^{-m})^k(\alpha^m\beta^m)^k$ for some $m$ and $k$ large.
More precisely, in case \ref{item nonseparated case}, consider $V_2=(\xi_2, a_2^-)$ small enough so that it is separated from $b$ by $l_1$, and similarly choose $U_2= (\eta_2, b^+)$ small enough so that it is separated from $a$ by $l_n$. In case \ref{item prong case}, we take $V_2$ to be sufficiently close to $\cF^-(a)$ so that at least two rays of $\cF^-(p)$ lie between $V_2$ and $\cF^-(b)$, on the side of the face of $l$ containing these endpoints, and similarly choose $U_2$ so that it is between $\cF^+(b)$ and the face of $\cF^+(p)$ meeting $\cF^-(b)$.

By Lemma \ref{lem: non corner partially linked}, we can pick $m$ large enough so that $V_2$ contains the attractor for $\alpha^m \beta^{-m}$ and $U_2$ the repeller of $\alpha^m \beta^m$. Therefore, picking any $k$ large enough, we deduce that $g_{m,k} (\partial P \setminus U_2) \subset V_2$ and $g_{m,k}^{-1}(\partial P \setminus V_2) \subset U_2$. (See Figure \ref{fig: product gk}.)

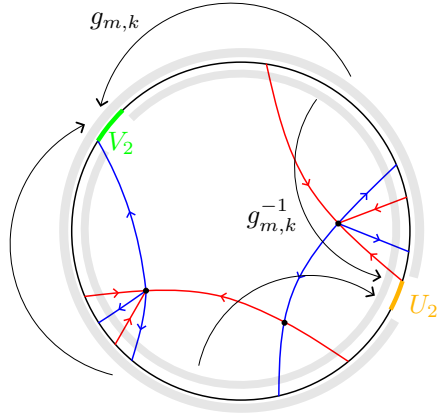
\begin{figure}[h]
\centering
\labellist
  \small

  \pinlabel {\color{green} $V_2$} [tl] at 90 220
  \pinlabel {\color{myorange} $U_2$} [tl] at 292 111
  \pinlabel { $g_{m,k}$} [br] at 120 280
    \pinlabel { $g_{m,k}^{-1}$} [tr] at 225 175
\endlabellist

\includegraphics[width=0.5\linewidth]{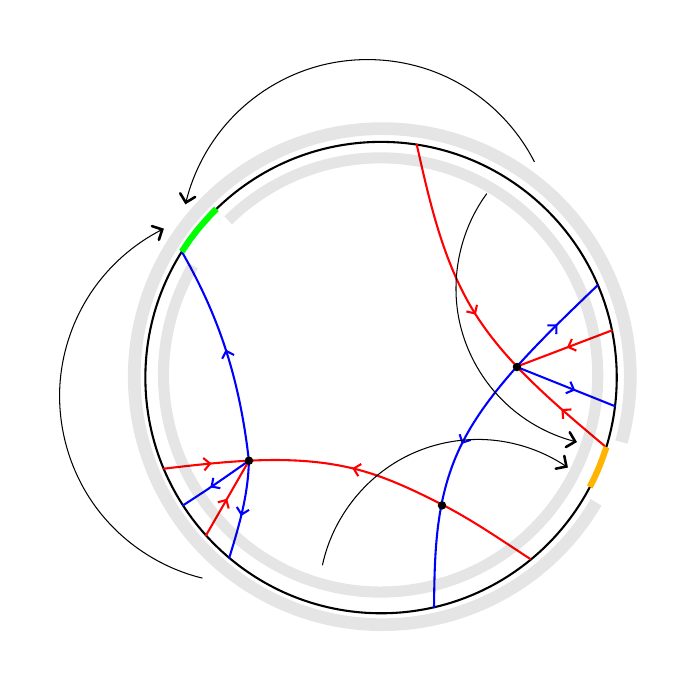}
    \caption{Action of the element $g_{m,k}$}
    \label{fig: product gk}
\end{figure}

To simplify notation, we fix $m$ and $k$ large as above, and write $g=g_{m,k}$. Note that in both case \ref{item nonseparated case} and \ref{item prong case}, $U_2$ and $V_2$ where chosen small enough so that they are separated by a leaf of both $\cF^+$ and $\cF^-$. Since $\mathrm{Fix}(g) \cap \partial P$ contains points in $U_2$ and $V_2$ and nowhere else, Lemma \ref{lem: connect fix sets} implies that $g$ acts freely. %

We now describe the axis of $g$.  
Suppose first we are in case \ref{item nonseparated case}.  %
Recall that $l_1$ is a leaf intersecting $\cF^-(a)$ and $l_n$ is a leaf non-separated with $l_1$ intersecting $\cF^-(b)$

The dynamics of $g$ implies that $l_1$ separates $g (l_n)$ from $l_n$ and $l_n$ separates $g^{-1}(l_1)$ from $l_1$.  Thus, both $l_1$ and $l_n$ lie in the axis, which means that the axis cannot be an embedded copy of $\R$, hence it is a union of segments as described.  
The argument in case \ref{item prong case} is similar, and shows that two faces of the prong lie in the axis, and that this divides the axis.  

To conclude the proof, it remains to show that $\mathcal{A}(g)$ is not contained in the projection of a chain of lozenges.  We will show this is the case provided that $m, k$ were chosen sufficiently large.  

We use the following basic claim about chains.
\begin{claim} \label{claim: chain lozenges and prong or non separated leaves}
Let $\cC$ be a maximal chain of lozenges. If $p$ is a prong singularity and not a corner of $\cC$, then $\cC$ is contained in 1, 2, or 3 adjacent quadrants of $p$. 
Similarly, if $s_1, s_2, \ldots$ is a union of nonseparated leaves that are not sides of $\cC$, then $\cC$ intersects at most one $s_i$. 
\end{claim}
\begin{proof}
First we treat the prong case.  Suppose $\cC$ does not have $p$ as a corner, and is not contained in a single quadrant. Then some side of a lozenge intersects a ray $r$ of $p$.  Let $q$ be the point of $\cC \cap r$ closest to $p$. 
If $q$ is an accumulation point of leaves of sides of $\cC$, then $\cC$ is a scalloped region so contained in two quadrants.  Otherwise, $q$ is the on the side of a lozenge $L$, let $c$ denote the corner on this side.  Since $q$ is a closest point, and $p$ cannot be contained in a lozenge as it is singular, there is no adjacent lozenge to $L$ on the other side of $q$ with corner $c$.  Thus $\cC$ is contained in the union of the other quadrants of $c$ (those not containing $p$) and thus meets at most $3$ quadrants of $p$. 

For the second statement, consider $s_1, s_2$ two non separated leaves and suppose that $\cC$ intersects $s_1$. We will show that $\cC$ cannot intersect $s_2$. To fix notations, assume that $s_1,s_2\in \cF^+$. Let $l_1\in \cF^-$ be the leaf making a perfect fit with $s_1$ and separating $s_1$ from $s_2$. We choose an orientation of $s_1$ towards the perfect fit with $l_1$ and call $q$ the supremum of $\cC \cap s_1$ for that orientation. Note first that $q$ cannot be infinite as it would force $l_1$, and thus $s_1$ and $s_2$, to contain sides of $\cC$. If $q$ is an accumulation point, then, as above $\cC$ must define a scalloped region $U$ and $\cF^-(q)$ separates $U$ from $l_1$, and therefore from $s_2$. So the conclusion of the claim holds in that case. So we may assume that $q$ is on a side of a lozenge $L$ in $\cC$. Let $c$ be the corner of $L$ on the same side of $s_1$ as $l_1$. As $q$ is defined as the supremum of $\cC \cap s_1$, we get as above that there are no lozenges in the quadrant of $Q$ that contains the perfect fit between $l_1$ and $s_1$.  This proves the claim.

\begin{figure}[h]
    \centering
    \labellist
    \small
    \pinlabel{$p$} [r] at 100 101
    \pinlabel{$q$} [b] at 129 72
    \pinlabel{$c$} [b] at 116 57
    \pinlabel{\color{red} $r$} [tl] at 150 51
    \pinlabel{$\cC$} [tr] at 58 36

    \pinlabel{{\color{red}$s_3$}} [b] at 305 158
    \pinlabel{{\color{red}$s_2$}} [b] at 390 158
    \pinlabel{{\color{red}$s_1$}} [b] at 493 158
    
    \pinlabel{$\cC$} at 499 210
    \pinlabel{$q$} [br] at 446 162
    \pinlabel{$c$} [br] at 446 108
    
    \endlabellist
    
    \includegraphics[width=1\linewidth]{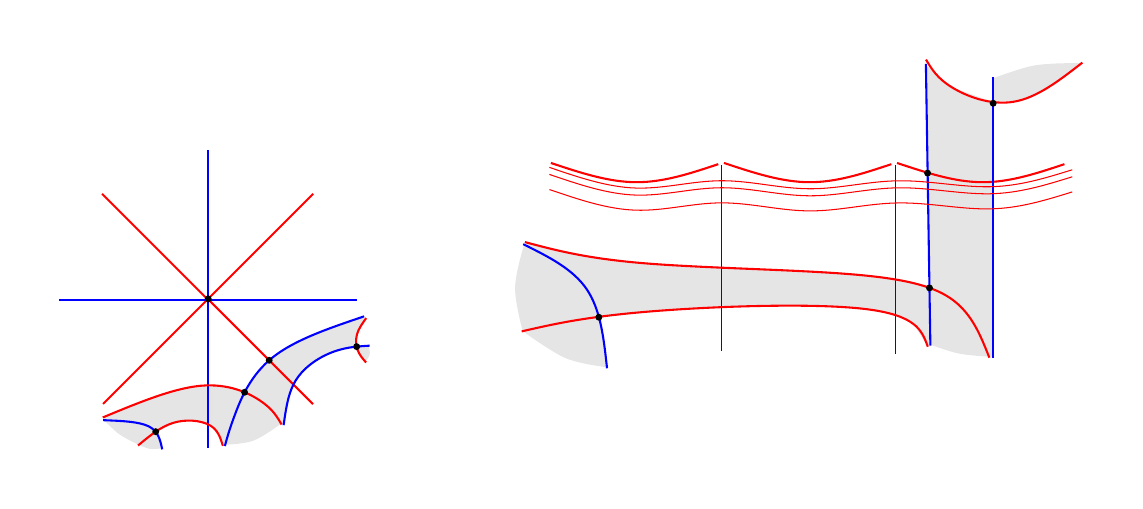}
    \caption{Chain of lozenge crossing a prong singularity (left) and nonseparated leaves (right)}
    \label{fig: nonsep to prong}
\end{figure}
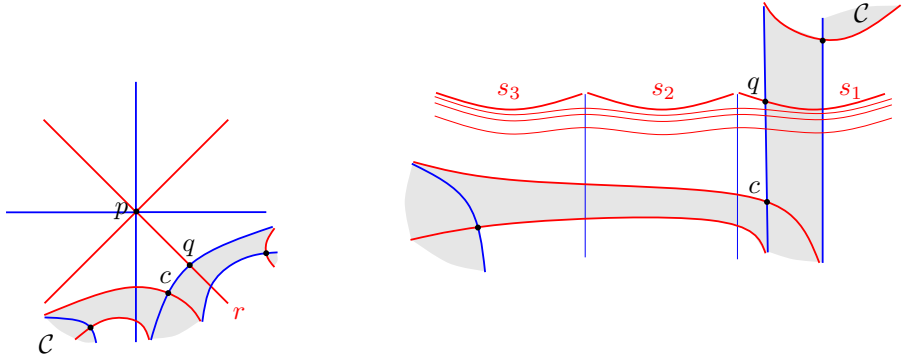

\end{proof}

Returning to the proof of the proposition, considering first case \ref{item nonseparated case}, suppose that $\cC$ is some chain of lozenges that intersects both $l_1$ and $l_n$.  
By Claim \ref{claim: chain lozenges and prong or non separated leaves}, 
$l_1$ and $l_n$ contain sides of lozenges in $\cC$.  
Since $l_1, \ldots l_n$ are nonseparated, they are sides of lozenges in a chain, fixed by some nontrivial element of $G$ (this is Proposition \ref{prop:nonseparated_on_chain}).  Thus the maximal chain containing $\cC$ has a side on each $l_i$. 

In summary, for any element $g_{m,k}$ constructed as above with $l_1$ and $l_n$ in its axis, and any chain $\cC$ whose projection contains the axis of $g_{m,k}$, the maximal chain containing $\cC$ is the {\em unique} maximal chain (independent of $m, k$ and of $\cC$) containing the non-separated leaves as sides of lozenges.  Let $\cC'$ denote this maximal chain, and $h \in G$ the element fixing all corners or $\cC'$ (which exists by Proposition \ref{prop:nonseparated_on_chain}). 

Since $a$ and $b$ were chosen to be non-corners, the leaves $\cF^\pm(a)$ and $\cF^\pm(b)$ do not form sides of lozenges in $\cC$. Thus, these leaves are not fixed by $h$ (by Theorem \ref{thm_fixed_connected}), and hence the endpoints of these leaves are also not fixed by $h$. 

Since $\mathrm{Fix}(h) \cap \partial P$ is a closed set, there is some neighborhood of the endpoints of $\cF^+(a)$ in $\partial P$ which contains no fixed points of $h$, i.e., no accumulation point of ends of leaves of sides of lozenges in the maximal chain. Therefore, we could have chosen our initial interval $V_2$ small enough (and appropriately increase $m$ and $k$ to some $m', k'$), so that $V_2$ contains no fixed point of $h$. Thus, $g_{m',k'}$ cannot preserve the chain $\cC'$.  In particular, this implies that $\cA(g_{m',k'})$ (which is invariant under $g_{m',k'}$) is not contained in the projection of $\cC'$.  Since $\cC'$ is the unique maximal chain whose projection contains $l_1$, which is in $\cA(g_{m',k'})$, the conclusion follows. 

The case for a dividing prong in place of nonseparated leaves follows exactly as above, using Claim \ref{claim: chain lozenges and prong or non separated leaves} to show that there is a unique maximal chain whose saturation contains points on an axis of $\cA(g_{m',k'})$, for any $m', k'$ sufficiently large, and that chain cannot project to contain all of $\cA(g_{m',k'})$.\qedhere 
\end{proof}

Our next goal is to show that the elements of the form constructed in  Proposition \ref{prop: free element condition} are loxodromic in $X^+$:
\begin{prop} \label{prop_axis}
If $g$ is as in Proposition \ref{prop: free element condition} then $g$ acts loxodromically on $X^+$.  Furthermore, there exist examples of such elements whose axes have distinct endpoints on the Gromov boundary $\partial X^+$.  In particular, $X^+$ is not elementary.
\end{prop}

 The first step is to show that the axes furnished by Proposition \ref{prop: free element condition} correspond to bi-infinite paths in $X^+$.  After this, we will show that we can build axes with distinct endpoints on $\partial X^+$ to complete the proof.  
 
Going forward, we let $d_{X^+}$ denote distance in $X^+$, where all edges in $X^+$ are given length 1.
\begin{lem} \label{lem: blocks are long}
Suppose that $\llbracket x,y\rrbracket = \bigsqcup_{i=1}^n [x_i,y_i]$.  Then $d_{X^+}(x,y) \geq n-1$. 

Analogously, suppose $x, y$ lie on an embedded interval in $\Lambda^+$ and $[x,y]$ can be decomposed as $[x,x_1] \cup[x_1,x_2] \cup... [x_n,y]$ , where each $x_i$ is a dividing prong, dividing $x_{i-1}$ from $x_{i+1}$ (or in the case of the prong $x_1$, dividing $x$ from $x_2$, and similarly for the prong $x_n$). Then $d_{X^+}(x,y) \geq n$. 
\end{lem}

\begin{proof}
Let $\gamma$ be a geodesic in $X^+$ from $x$ to $y$. Then, $\pi_{\Lambda^+}(\gamma)$ contains $\llbracket x,y\rrbracket$.  Since no leaf of $\cF^-$ can intersect two distinct intervals $[x_i, y_i]$ in the decomposition, the $\cF^+$-saturation of any edge $e$ of $\gamma$ meets at most one segment $[x_i, y_i]$.  Thus, the number $n$ of segments gives a lower bound to the number of edges in $\gamma$, and thus $\gamma$ has at least $n+1$ vertices.  

Similarly, suppose now $\llbracket x,y\rrbracket = [x,y]$ can be decomposed as $[x,x_1] \cup [x_1,x_2] \cup... [x_n,y]$ where the $x_i$ are dividing prongs.
Remark \ref{rmk: separating prong separates} shows $d(x_i, x_{i+2})\geq2$, so each interval $[x_i,x_{i+1}]$ must contain at least one vertex of any  path from $x$ to $y$. 
\end{proof}

As a consequence we have: 
\begin{coro}
    
$X^+$ has infinite diameter.  Moreover, if  $g$ is the element constructed in the proof of Proposition \ref{prop: free element condition}, then $g$ acts loxodromically on $X^+$. 

\end{coro}

Thus, to complete the proof of Proposition \ref{prop_axis}, we need only show that the action of $G$ is nonelementary, meaning that there exist two quasi-axes 
which are not finite distance apart.   We give one more short lemma which simplifies the argument (and also partially explains why we require axes to not be contained in the saturation of a chain of lozenges) and then finish the proof.

\begin{lem}\label{lem: unique fixed point or two fixed points}
Suppose $g$ acts loxodromically on $X^+$. Then the action on $\partial P$ is of one of two possible types:
\begin{enumerate}[label=(\roman*)]
    \item $g$ fixes a unique point $\xi\in \partial P$, $\xi$ is the end of (infinitely many) leaves of $\cF^\pm$, and the axis of $g$ is contained in the saturation of a $g$-invariant chain of lozenges;  
    \item $g$ fixes exactly two points $\xi^+,\xi^-\in \partial P$, neither of which are the ends of any leaves of $\cF^\pm$.
\end{enumerate}
\end{lem}

\begin{figure}[h]
    \centering
    \labellist
    \small
    \pinlabel{$\xi$} [br] at 73 264
    \pinlabel{$g$} [tl] at 147 142

    \pinlabel{$\xi^-$} [r] at 333 157
    \pinlabel{$\xi^+$} [l] at 604 154

    \pinlabel{$g$} [b] at 488 163
    
    \endlabellist
    \includegraphics[width=0.9\linewidth]{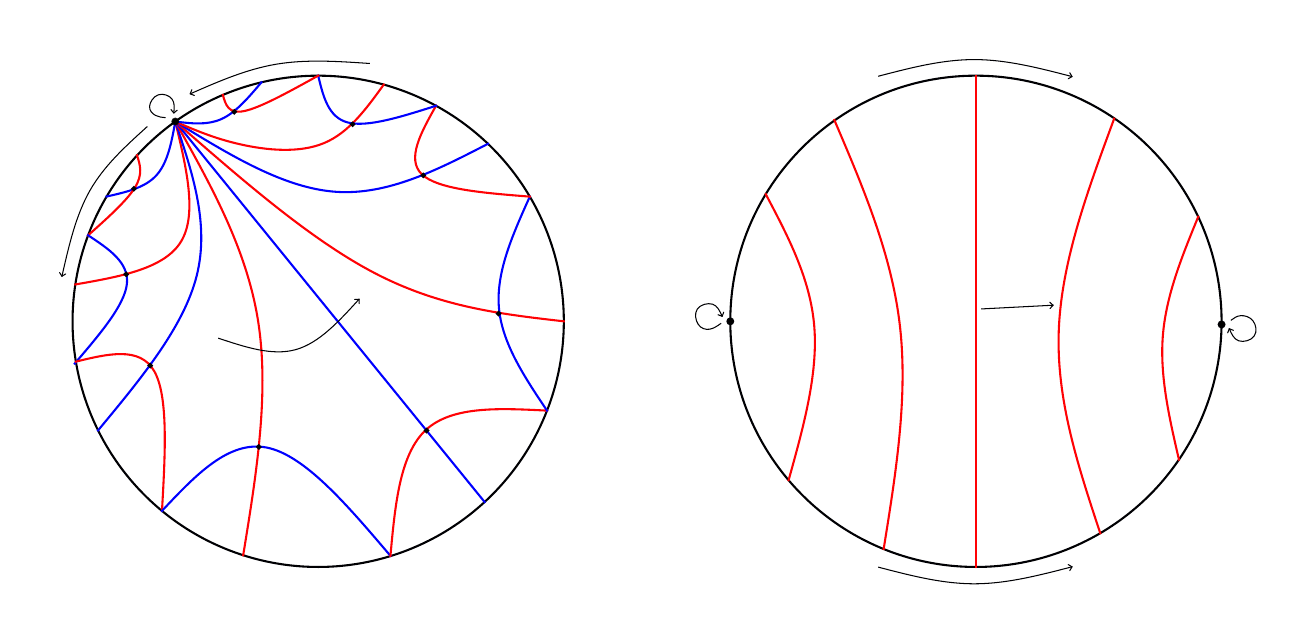}
    \caption{Two types of loxodromic element of Lemma \ref{lem: unique fixed point or two fixed points}}
    \label{fig: loxo 1 2 fixed point}
\end{figure}
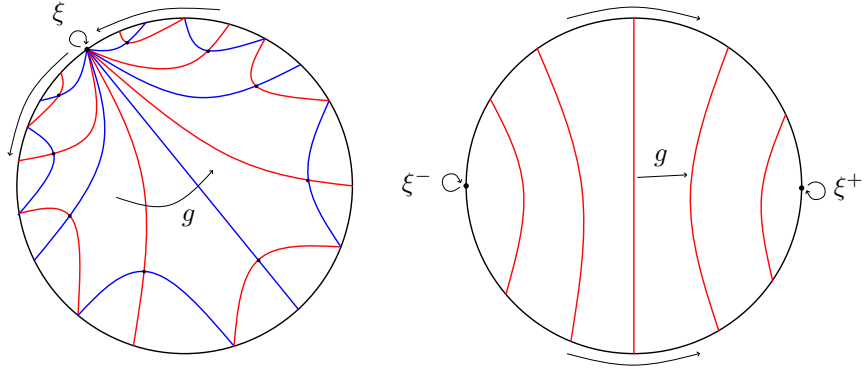

\begin{proof}
Since $g$ acts loxodromically on $X^+$, it acts freely on $P$ and therefore (by Axiom \ref{Axiom_A1}) has an axis $\cA(g)$ in $\Lambda^+$. Let $l\in \Lambda^+$. Note that $g^n(l)$ must escape every compact as $n\to \pm \infty$, as otherwise we would have that $g^n(l)$ converge to a union of non-separated leaves, and so either $g$ would be in the stabilizer of a scalloped region, or some power of $g$ would fix one of these limit leaves, and thus fix a point in $P$. In either case, this contradicts that $g$ is loxodromic on $X^+$. Thus the ends of $g^n(l)$ on $\partial P$ must converge to a single point $\xi^+$ as $n\to \infty$ and a point $\xi^-$ (possibly equal to $\xi^+$) as $n\to -\infty$. 

By definition of axis, $l$ separates $g(l)$ from $g^{-1}(l)$, so we have two possible cases: Either $g^{-1}(l)$ and $g(l)$ have a common endpoint $\xi$ on $\partial P$ or they do not. In the first case, since $l$ separates $g(l)$ from $g^{-1}(l)$, $\xi$ must also be an endpoint of $l$, and therefore $\xi$ must be the endpoint of $g^n(l)$, for all $n$. In particular, $\xi$ is fixed by $g$, and all the $g^n(l)$ are part of the same $g$-invariant chain of lozenges, so $\xi^+=\xi=\xi^-$ and we are in the first case of the lemma.

Otherwise, $l$ and $g^2(l)$ do not share endpoints, so the limits $\xi^\pm$ are in two distinct components of $\partial P \smallsetminus l$. Moreover, neither of $\xi^\pm$ can be the end of a leaf as we would have a contradiction from the first part (using the fact that $g$ is loxodromic for $X^+$ if and only if it is loxodromic on $X^-$).
\end{proof}

\begin{proof}[Proof of Proposition \ref{prop_axis}]
Given the above, it suffices to find $g, h$, as in Proposition \ref{prop: free element condition} whose quasi-axes have distinct endpoints on the Gromov boundary $\partial X^+$.  

Let $g$ be an element constructed using Proposition \ref{prop: free element condition}. 
For simplicity, we treat first the case where $g$ is in case \ref{it: non sep} of Proposition \ref{prop: free element condition}.  Let $\xi^+$ denote the attracting fixed point of $g$ on $\partial P$ and $\xi^-$ the repelling fixed point.  By Lemma \ref{lem: unique fixed point or two fixed points}, $\xi^+$ and $\xi^-$ are not endpoints of any leaves.  

To construct $h$, we proceed as in the proof of Proposition \ref{prop: free element condition}, taking $a'$ and $b'$ elements with partially linked noncorner fixed points.  By using large powers of $a'$ and $b'$ in the construction, we can ensure that the axis of $h$ in $P$ has endpoints as close as we like to ends of leaves of $a'$ and $b'$, and thus are disjoint from $\{ \xi^+ , \xi^-\}$. 
We claim that $h$ also acts on $X^+$ with quasi-axis having distinct endpoints from that of $g$.

To show this, let $l_1, l_n$ be nonseparated leaves in the axis of $g$, keeping the notation from the proof of Proposition \ref{prop: free element condition}. 
Consider all of their translates by powers of $g$.  Since the axis of $h$ does not share endpoints $\{ \xi^+ , \xi^-\}$, all the leaves in the axis of $h$ sufficiently far towards the attracting fixed point for $h$ are eventually in one connected component of $P \setminus \bigcup_{n \in \mathbb{Z}} g^n(l_1 \cup l_n)$.  (See Figure \ref{fig_axes_g_h}.)
Let $l_h$ a leaf of $\mathcal{A}(h)$ in this connected component.  Then 
if $N$ is sufficiently large, then $g^N(l_1)$ and $g^N(l_n)$ separate $l_h$ from $g^{N+1}(l_1)$ and $g^{N+1}(l_n)$, 
and similarly $g^{-N}(l_1)$ and $g^{-N}(l_n)$ separate $l_h$ from $g^{-N-1}(l_1)$ and $g^{-N-1}(l_n)$. 
Fix some such $N$ where this holds.  
Then, for any $K>0$, any realization in $P$ of a geodesic path from $l_h$ to $g^{N+K}(l_1)$ in $X^+$ must intersect all leaves $g^{N+i}(l_1)$ and $g^{N+i}(l_n)$, for $i<K$.  Since no leaf of $\cF^-$ can simultaneously meet $l_1$ and $l_n$, it follows the length of the geodesic path is at least $K-1$. 

A similar argument holds for leaves close to the repelling fixed point for $h$.  This shows that all ends of the axes of $g$ and $h$ are unbounded distance apart, and hence they have distinct endpoints on the Gromov boundary of $X^+$.

The argument when $g$ is in case (2), with dividing prongs, follows in exactly the same way using the fact that edge leaves of a trace cannot cross multiple dividing prongs.  
\end{proof}

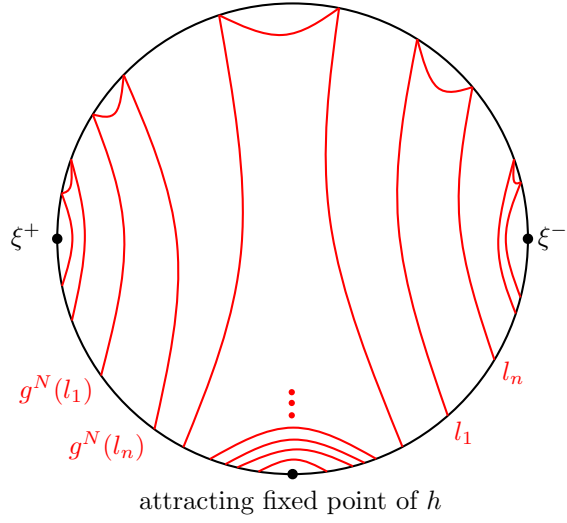
\begin{figure}[h]
    \centering
\labellist
  \small
\pinlabel {$\xi^+$} [r] at 27 145
\pinlabel {$\xi^-$} [l] at 258 145
\pinlabel{\color{red} $l_1$} [tl] at 219 59
\pinlabel{\color{red} $l_n$} [tl] at 242 88

\pinlabel{\color{red} $g^N(l_1)$} [tr] at 51 81
\pinlabel{\color{red} $g^N(l_n)$} [tr] at 76 54

\pinlabel {attracting fixed point of $h$} [t] at 143 27
\endlabellist

    \includegraphics[width=.6\linewidth]{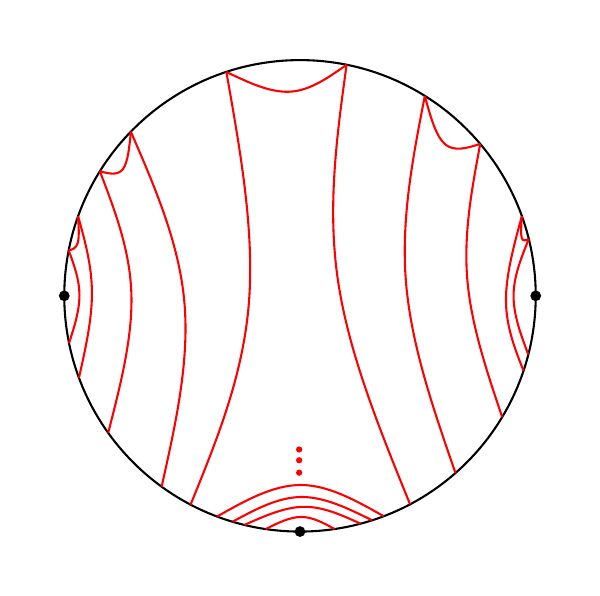}
    \caption{One connected component of $P \setminus \bigcup_{n \in \mathbb{Z}} g^n(l_1 \cup l_n)$ contains the attractor of $h$.}
    \label{fig_axes_g_h}
\end{figure}

\subsection{Classification of isometries: Proof of Theorem \ref{thm_no_parabolics}} \label{subsec: classification isometries}
Recall Theorem \ref{thm_no_parabolics} states that 
elements acting loxodromically on $X^+$ act freely on $P$, those acting elliptically either fix a point or preserve a scalloped region, and there are no parabolic elements.  We now give the proof.  
  
\begin{proof}[Proof of Theorem \ref{thm_no_parabolics}]
First, observe that by Axiom \ref{Axiom_A1}, an element $g \in G$ fixes a point of $P$ if and only if it fixes a vertex in $X^+$.  Thus, we need only analyze the situation of elements that act freely on $P$. 
If the axis of $g$ in $\Lambda^+$ is a pseudo-line, then by Lemma \ref{lem: blocks are long}, the asymptotic translation length of $g$ is positive.  
If the axes $\mathcal{A}^+(g)$ and $\mathcal{A}^-(g)$ in $\Lambda^+$ and $\Lambda^-$ are both homeomorphic to $\mathbb{R}$, then we can apply Theorem 3.5.2 from \cite{barthelmePseudoAnosovFlowsPlane2026}.  This theorem says that either $P$ is trivial (in which case $g$ acts elliptically), or $g$ preserves a scalloped region (whose interior has diameter $1$ in $X^+$, hence $g$ has a bounded orbit), or $g$ preserves a ``skew-like structure" in $P$, in the sense described below. 
Precisely, there is a $g$-invariant set 
\[\Omega := \{p \in P : p \in l^+ \cap l^- \text{ for some } l^\pm \in \mathcal{A}^\pm(g) \}\] 
and well-defined functions $s^\pm$ and $i^\pm$ from $\mathcal{A}^\pm(g)$ to $\mathcal{A}^\mp(g) \cong \R$, given by 
$s^\pm(l) = \sup \{ l' : l \cap l' \neq \emptyset \}$ and $i^\pm(l) = \inf \{ l' : l \cap l' \neq \emptyset \}$, which are {\em monotone} and {\em finite}. 

Assume we are in this third case and 
fix $v \in \cA^+(g)$.   Since $i^+$ and $s^+$ are monotone and $g$ translates along the axis, this means that there exists $k$ such that $i^+(g^k(v)) > s^+(v)$.  Thus, no leaf of $\Lambda^-$ intersects both $v$ and $g^k(v)$.  If $v = v_0, e_1, v_1, \ldots v_j = g^{kn}(v)$ is any path in $X^+$ from $v$ to $g^{kn}(v)$ for $n>0$,  then its projection to $\Lambda^+$ contains the interval along the axis, and the vertex $v_1$ must have $i^+(v_1) < s^+(v_0)$ as witnessed by the edge $e_1$.   By construction, $i^+$ and $s^+$ commute with the action of $g$.  Using this and monotonicity, we conclude that $j \geq n$, and thus the asymptotic translation length of $g$ is at least $1/k$.   

The only case not treated so far is when $\mathcal{A}^+(g)$ is homeomorphic to $\mathbb{R}$ but $\mathcal{A}^-(g)$ is a pseudo-line. But as $X^-$ and $X^+$ are quasi-isometric (Proposition \ref{prop: X+ X- QI}) $g$ acts loxodromically on one if and only if it acts loxodromically on the other, so the first part of the proof applies switching the $+$ and $-$ signs.
\end{proof}

\section{WPD elements and proof of Theorem \ref{thm_generic_plane_version}} \label{sec: WPD}
In this section we conclude the proof of Theorem \ref{thm_generic_plane_version}, showing that for any Anosov-like group acting on a non-skew, nontrivial plane, the induced action on $X^+$ has elements with the {\em weak proper discontinuity property}, also known as {\em WPD}. 

\begin{defi}[WPD]\label{def: WPD}
    Assume $G$ acts by isometries on a hyperbolic metric space $X$.  
    A loxodromic element $g$ is \emph{WPD} if there exists a point $p \in X$ such that for every $\epsilon > 0$, there exists $n \in \mathbb N$ with the property that the set
    \[
        \bigl\{\, h \in G \mid d(p,hp) < \epsilon \text{ and } d(g^n p, h g^n p) < \epsilon \,\bigr\} 
    \]
    is finite.
\end{defi}

To prove Theorem \ref{thm_generic_plane_version}, we will actually show the following: 
\begin{prop} \label{prop:WPD}
Any element satisfying the conditions of Proposition \ref{prop: free element condition} is WPD. 
\end{prop}

In order to do this, we first formalize and make more precise the kind of decomposition of pseudo-intervals that was used in the proof of Lemma \ref{lem: blocks are long}. This allows us to treat the presence of nonseparated leaves and of dividing prongs along an axis in a unified way.  
We then prove a series of lemmas giving control on the displacement of long sub-segments of an axis.  
Many of theses statements are not specific to axes, but apply to more general subsets of leaf spaces.  For this we make the following definition:

\begin{defi}
    A subset $\mathcal{A} \subset \Lambda^+$ is a {\em pseudo-line} if, for any $x, y$ in $\mathcal{A}$, the complement of $\llbracket x, y \rrbracket$ in $\mathcal{A}$ has two connected components, each of which leaves all compact sets in $\Lambda^+$.
\end{defi}
A pseudo-line may be a properly embedded copy of $\mathbb{R}$, in which case we call it an embedded line, or a bi-infinite union of intervals $[x_i, y_i]$ with $x_i$ nonseparated with $x_{i+1}$, or the union of a (finite or infinite) number of such intervals with (two, or one) properly embedded infinite rays.  

\begin{lem}[Discreteness of dividing prongs] \label{lem_discrete_dividing}
 If $\cA$ is an embedded line in $\Lambda^+$ 
 then there are finitely many dividing prongs in any compact subset of $\cA$.
\end{lem}
\begin{proof}
Suppose for contradiction that dividing prongs $p_i$ along $\mathcal{A}$ accumulate on some leaf $l$, passing to a subsequence, we can assume that $\cF^+(p_i)$ converges to $l$.  This means that there is a face $f_i$ of $\cF^+(p_i)$ so that $f_i$ converges to $l$ in $P$ uniformly on compact sets. 

Since the set of singularities is discrete in the plane, the points $p_i$ escape all compact sets, thus there is in fact a ray $r_i$ based at $p_i$ of $f_i$ so that the $r_i$ converge to $l$ (uniformly on compact sets) in $P$.  Note that there cannot be any other ray $r'_i$ of $p_i$ which converges to $l$, since this would force the $\cF^-$ ray(s) between $r_i$ and $r_i'$ to converge as well, contradicting transversality of the foliations.  Thus, for all sufficiently large $i$, some (necessarily unique) ray of $\cF^-(p_i)$ crosses $l$, and thus also crosses the rays $r_n$ for $n>i$.  This configuration means that such $p_n$ cannot be dividing prongs, as the ends of $\mathcal{A}$ are separated only by a single quadrant of $p_n$.  
See Figure \ref{fig: prong discrete} where the single quadrant is shaded.
\end{proof}

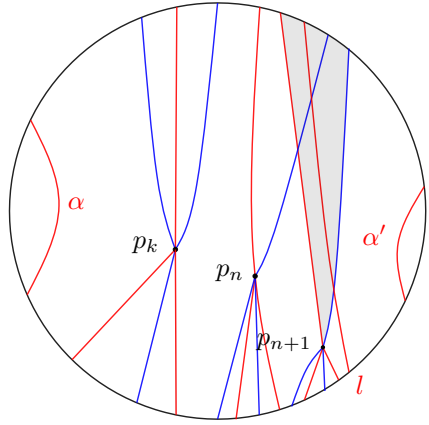
\begin{figure}[h]
    \centering
\labellist
  \small
  \pinlabel{$p_k$} [tr] at 137 160
  \pinlabel{$p_n$} [tr] at 196 140
  \pinlabel{$p_{n+1}$} [tr] at 243 92
 \pinlabel{\color{red} $l$} [tl] at 266 63
\pinlabel{\color{red} $\alpha$} [l] at 64 179
\pinlabel{\color{red} $\alpha'$} [r] at 297 157

\endlabellist
    
    \includegraphics[width=0.5\linewidth]{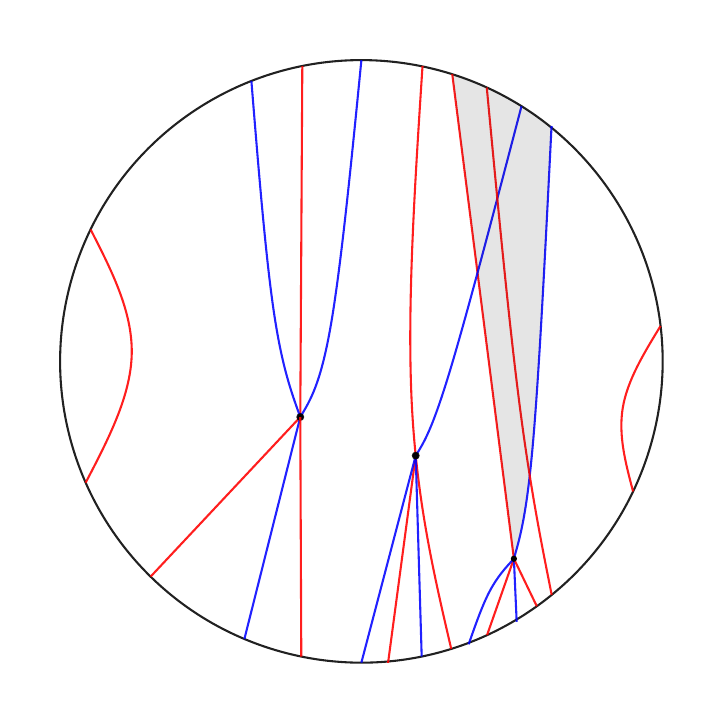}
    \caption{$p_k$ divides $\alpha$ and $\alpha'$ on $\cA$, but $p_n$ accumulating on $l$ eventually cannot.}
    \label{fig: prong discrete}
\end{figure}

Lemma \ref{lem_discrete_dividing} allows us to imitate the canonical decomposition $\cA = \sqcup [x_i, y_i]$ where the $[x_i,y_i]$ are maximal intervals (in the case where there are nonseparated leaves, so $x_i$ nonseparated with $y_{i-1}$) to obtain a similar canonical decomposition when $\cA \cong \mathbb{R}$ but contains dividing prongs. 

\begin{defi}[Block decomposition]
    Let $\cA$ be the axis of an element $g$.  
The {\em block decomposition of $\cA$} is given by $\cA = \bigcup_{i \in \mathbb{Z}} [x_i, y_i]$, where these intervals have pairwise disjoint interiors and 
either: 
\begin{itemize}
    \item for all $i$, $[x_i,y_i]$ is an interval and $x_i$ is nonseparated with $y_{i-1}$, or 
    \item $\cA \cong \mathbb{R}$, and for all $i$, $x_i$ is a dividing prong, $x_i = y_{i-1}$,  each dividing prong appears as one of the $x_i$ in the decomposition.   
\end{itemize}
\end{defi}
Observe that the block decomposition of the axis of $g$ is automatically $g$-invariant. We will typically denote $[x_i, y_i]$ by $B_i$.  When $\cA \cong \mathbb{R}$ and has no dividing prongs, this decomposition is empty, but for all of the elements that we will work with, the decomposition is always nonempty and bi-infinite.

To motivate the next lemma, recall that if a pseudo-interval $\llbracket x, y \rrbracket$ has decomposition $\llbracket x, y \rrbracket = \sqcup [x_i, y_i]$, and 
$\llbracket x', y' \rrbracket \subset \llbracket x, y \rrbracket$ is a sub-pseudo interval, then the intervals in the decomposition of the pseudo-interval $\llbracket x', y' \rrbracket$ are of the form 
$[x', y_j], [x_{j+1}, y_{j+1}] \ldots [x_k, y']$ for some $j \leq k$.   The next lemma says that the same is true of blocks, up to some ambiguity at the initial and terminal intervals in the case of dividing prongs.  
\begin{lem}[Block decomposition of sub-pseudo-intervals] \label{lem: sub broken axis}
Suppose $\llbracket x', y' \rrbracket$ is a subinterval of an axis $\cA$ as above, with $x' \in B_i$ and $y' \in B_{n}$, where $n\geq i+2$.  Then 
$\llbracket x', y' \rrbracket$ has a canonical block decomposition into intervals
\[ [x', y_j] \cup [x_{j+1}, y_{j+1}] \cup  \ldots \cup [x_k, y'] \] where $|j-i|\leq 1$ and $|n-k|\leq 1$, 
with the property that either 
\begin{itemize}
    \item for all $i$, $x_i$ is nonseparated with $y_{i-1}$, or 
    \item for all $i$, $x_i$ is a dividing prong for $[x',y']$, $x_i = y_{i-1}$, and each prong dividing $x'$ and $y'$ appears as one of the $x_i$ in the decomposition. 
\end{itemize}
\end{lem}

\begin{proof}
We already observed this to be true in the case of nonseparated leaves.  For prongs, the only nuance is that if $x' \in [x_i,y_i]$, then it may no longer be the case that $y_i$ is a prong that divides $x'$ from the point $y'$, since $x'$ may not be contained in a single quadrant of $y_i$.  However, the construction of dividing prongs implies that $x'$ {\em is} contained in a quadrant of $x_{i+1} = y_{i+1}$, which will then separate it from $y'$. 
\end{proof}
We will call the decomposition given by Lemma \ref{lem: sub broken axis} the {\em induced block decomposition} of $\llbracket x', y' \rrbracket$.  
This will play a useful role at the end of the proof of Proposition \ref{prop:WPD}.

We first show that no element can simultaneously fix points in far away blocks along an axis as in Proposition \ref{prop: free element condition}: 
\begin{lem} \label{lem cant fix two far blocks}
    Suppose $g$  satisfies the conditions in Proposition \ref{prop: free element condition}, and let $B_i$ be the block decomposition of its axis.  Suppose that $h$ is a nontrivial element of $G$ which fixes a leaf $z \in B_i$ and a leaf $w \in B_j$.  Then $B_i$ and $B_j$ both lie in the same fundamental domain for the action of $g$, i.e. $|j-i| < t$. 
\end{lem}
\begin{proof}
If $h$ fixes two distinct leaves $z$ and $w$, then by Theorem \ref{thm_fixed_connected} they each contain sides of lozenges in a common chain of lozenges.  
Claim \ref{claim: chain lozenges and prong or non separated leaves} implies that, either $z$ and $w$ are contained in the same block, or for any endpoint of a block interval $x_i$ or $y_i$ that is between $z$ and $w$ on the axis, the leaf $x_i$ (respectively, $y_i$) is the side of a lozenge in the chain.   

If $z$ and $w$ are not contained within the same fundamental domain, then (up to switching their names) we have $z \in B_i$ and $w \in B_{i+t+n}$ for some $n\geq 0$.  Thus, $y_i$, $x_{i+1}, \ldots ,x_{i+t}$ and $y_{i+t} = g(y_i)$ are all leaves in sides of this chain of lozenges. If $\cC$ is a chain which shares all of these sides, then $g \cC$ has a side on $g(y_i)$, so is part of the maximal chain containing $\cC$.  Moreover, $g \cC$ has sides on the leaves $g(x_{i+1}),\ldots, g(x_{i+t})$. 

Proceeding iteratively, we see that the maximal chain of lozenges containing $z$ and $w$ has a side contained in each leaf $x_i$ and $y_i$.  Since chains project to connected subsets of $\Lambda^+$, we conclude that the whole axis $\cA$ is contained in the projection of a chain of lozenges, contradicting our assumption. 
\end{proof}

In order to get the WPD condition, we will show that if WPD fails, then we can find an element fixing points in far away blocks, thus contradicting 
Lemma \ref{lem cant fix two far blocks}.  For this, we first need an estimate on intervals with large overlap (Lemma \ref{lem: big overlap}) and a notion of closest point projection.  

\begin{defi}
 Let $\mathcal{A}$ be a pseudo-line in $\Lambda^+$ and let $x \in \Lambda^+$.  The projection $p_{\cA}(x) \subset \cA$ of $x$ to $\cA$ is defined to be either 
 the unique leaf (which is not nonseparated with any other in $\cA$), or the unique union of two nonseparated leaves, such that for all $y\in\cA(g)$, the pseudo-interval $\llbracket x,y \rrbracket$ intersects $p_\cA(x)$.  
\end{defi}

Recall that $\Lambda^+$ is a non-Hausdorff tree.  This definition of projection to $\cA$ in $\Lambda^+$ is  the analog of the closest point projection to an embedded copy of $\R$ in a standard tree. 
The next lemma shows that $p_\cA$ is also (coarsely, since $p_\cA(x)$ can have diameter 2) a closest point projection in $X^+$. 

\begin{lem} \label{lem: projection minimize distance}
For any $x \in \Lambda^+$, and any $y$ in a pseudo-line $\cA$, we have $$ d_{X^+}(x,p_\cA(x)) \leq d_{X^+}(x,y).$$
\end{lem}
\begin{proof}
Suppose $\gamma$ is a geodesic path in $X^+$ from $x$ to a leaf $y \in \cA$, then $\pi_{\Lambda^+}(\gamma)$ contains $\llbracket x, y \rrbracket$.  This means that some subpath of $\gamma$ projects to the sub-pseudo-interval $\llbracket x, p \rrbracket$ where $p \in p_\cA(x)$.  It follows that there is a path from $x$ to $p$ of length bounded above by the length of $\gamma$, giving the desired bound.   
\end{proof}

\begin{lem}[Large overlap] \label{lem: big overlap}
Let $\cA$ be a pseudo-line, let $\epsilon>0$ let $\llbracket a, b \rrbracket \subset \cA$, 
and suppose $\dX(a, b) > 4 \epsilon + 5$, and that
$\dX(a, h(a)) < \epsilon$ and $\dX(b, h(b)) < \epsilon$ for some $h \in G$.  
Then, there exists $J = \llbracket j_1, j_2 \rrbracket \subset \llbracket a, b \rrbracket$ 
such that $h(J) \subset  \llbracket a, b \rrbracket$ and satisfying: 
\[ \dX(j_1, a) \leq 2\epsilon +2, \quad \dX(j_2, b) \leq 2\epsilon +2, \]
\[ \dX(h(j_1), a) \leq 2\epsilon +2, \quad \text{ and } \dX(h(j_2), b) \leq 2\epsilon +2.\]
\end{lem}

\begin{proof}
    Let $J = \llbracket a, b \rrbracket \cap \llbracket h^{-1}(a), h^{-1}(b) \rrbracket$.  $J$ is a pseudo-interval since it is the intersection of two pseudo-intervals. 
    We show it is nonempty and has the desired properties.  

    First, by definition of projection, we can equivalently write $J = \llbracket a, b \rrbracket \cap \llbracket a', b' \rrbracket$ where $a' \in p_\cA (h^{-1}(a))$ and $b' \in p_\cA (h^{-1}(b))$.   We have 
    \[ \dX(a, a') \leq \dX(a, h^{-1}(a)) + \dX(h^{-1}(a), a') \leq \epsilon + \epsilon + 2\] 
    and similarly $\dX(b, b') \leq 2 \epsilon +2$.   In this estimate, we use that $h$ acts by isometries so $\dX(a, h^{-1}(a)) = \dX(h(a), a)$ and that $a'$ is (up to distance 2) a closest point in $\cA$ to $h^{-1}(a)$.  The same estimate holds for $b$.  
    Our assumptions on the distance between $a$ and $b$ in $X^+$ implies that either $a \in \llbracket a', b \rrbracket$ or 
    $a' \in \llbracket a, b \rrbracket$, and is distance (in $X^+$) at most $2 \epsilon +2$ from $a$.  Similarly, we have either that $b \in \llbracket a, b' \rrbracket$ or that $b$ is in $\llbracket a, b' \rrbracket$ but not too close to $a$, ensuring that $J$ is nonempty. 

    We also have that 
    \[ h(J) = \llbracket h(a), h(b) \rrbracket \cap \llbracket a, b \rrbracket \] 
    and can similarly write this in the form 
    $h(J) = \llbracket a, b \rrbracket \cap \llbracket a'', b'' \rrbracket$ where $a'' \in p_\cA(h(a))$ and $b''\in p_\cA(h(b))$.  The same estimate as before says that $\dX(a, a'') \leq 2\epsilon +2$ and $\dX(b, b'') \leq 2\epsilon + 2$.
    \end{proof}

With these tools we can now prove Proposition \ref{prop:WPD}.  
\begin{proof}[Proof of Proposition \ref{prop:WPD}]
    Let $g$ be as in Proposition \ref{prop: free element condition}, let $\cA$ be the axis of $g$ and $\cup B_i$ is its block decomposition. 
Let $\epsilon >0$ be given. 
    
    Let $m$ be the maximum diameter (in $X^+$) of a block $B_i$ in $\mathcal{A}$.  Since $g$ acts by isometries of $X^+$, translating blocks along the axis, this maximum is indeed finite.  By Lemma \ref{lem: blocks are long}, we have for all $k$,
    \[ (m+2)(k-1) + 2 \geq \dX(B_i, B_{i+k}) \geq k-1. \] 
    The factor $m+2$ appears because the distance from $y_i$ to $x_{i+1}$ is 2 when these are nonseparated leaves. 
    
In particular, there exists $n>0$ such that, for any $x \in B_i$ and $y \in \cA$, if $d_{X^+}(x, y) \leq 2\epsilon +3$ then $y \in B_{i+k(y)}$ for some $|k(y)| < n$.

Suppose for contradiction that there exist infinitely $a,b \in \cA$ with $d_X(a,b)> \max\{(2n+ t +1) (m+2) + 2, 4\epsilon +5\}$ and infinitely many $h \in G$ such that $d(a, h(a))< \epsilon$ and $d(b, h(b))< \epsilon$.  
Since $(m+2)(k-1) + 2 \geq \dX(B_i, B_{i+k})$, we may assume that $a \in B_0$ and $b \in B_N$, with $N > 2n+ t +1$.

For each such $h$, let $J_h=\llbracket p_h, q_h\rrbracket$ be the pseudo-interval given by Lemma \ref{lem: big overlap}.  Then as $d_{X^+}(a,p_h)\leq 2\epsilon + 2$ and $d_{X^+}(b,q_h)\leq 2\epsilon + 2$, there exists $0\leq i_h,i'_h,j_h, j'_h \leq n$ such that $p_h\in B_{i_h}$, $h(p_h)\in B_{i'_h}$, $q_h\in B_{N-j_h}$ and $h(q_h)\in B_{N-j'_h}$.

By Lemma \ref{lem: sub broken axis} the block decomposition is cannonical so any element taking a block into the axis sends it to another block.  By the pigeonhole principle, there exists $h_1 \neq h_2$ such that $i_{h_1} = i_{h_2},i'_{h_1}= i'_{h_2},j_{h_1}= j_{h_2}$ and $j'_{h_1}= j'_{h_2}$. In particular, taking $f=h_1\circ h_2^{-1}$, there exists $0\leq i,j \leq n$ such that $f(B_i) = B_i$ and $f(B_{N-j}) = B_{N-j}$.
So $f$ fixes the right hand side $y_i$ of $B_i$ and the left hand side $x_{N-j}$ of $B_{N-j}$. But $d_{X^+}(y_i, x_{N-j}) \geq N-j-i-1 \geq N-2n-1 > t$, which contradicts Lemma \ref{lem cant fix two far blocks}.

\end{proof}

Given Proposition \ref{prop:quasi-tree}, Proposition \ref{prop:WPD} ends the proof of Theorem \ref{thm_generic_plane_version}.
From this, we also deduce Corollary \ref{cor_acylindrically_hyperbolic}, i.e., the fact that a group admitting an Anosov-like action on a nonskew nontrivial plane is acylindrically hyperbolic, thanks to the following result of Osin: 
\begin{thm}[See Osin \cite{Osin2016Acylindrically}, Theorem 1.2]
A group $G$ admits a non-elementary acylindrical action on a hyperbolic space if and only if 
$G$ is not virtually cyclic and admits an action on a hyperbolic space such that at least one element is loxodromic and satisfies the WPD condition.
\end{thm}

\section{Graphs and metrics on the plane} \label{sec: metrics on plane}
Our choice to construct the graphs $X$, $X^+$ and $X^-$ above was just one of the many ways one can use the intersection pattern of leaves in a bifoliated plane to build a metric space on which the automorphisms of the bifoliated plane acts by isometries. In this section, we explain how one can recover quasi-isomorphic spaces using (discrete) metrics on the plane, and explain the parallel with $\mathrm{CAT}(0)$ cube complexes via the \emph{dualizable systems} of \cite{Petyt-Z24}.

Before doing this, we discuss another graph.  Its metric structure was previously studied by Barbot, but without any formalization as a combinatorial object.  As well as indicating the connection with Barbot's work, we will use this structure later in the proof of Theorem \ref{thm: genericity broken axis}. 

\subsection{Barbot's leaf space metric}

In \cite{Bar98}, Barbot considers a ``distance" on leaf spaces by counting the minimal number of intervals (or blocks in a pseudo-interval) between two points.  One can capture this data combinatorially, paralleling the construction of $X^\pm$, as follows: 

\begin{defi}\label{def_broken_axis_graph}
Define $\Gamma^\pm$ to be the graph whose vertices are leaves in $\cF^\pm$ and two vertices $x,y$ are connected by an edge if there exists a path in $P$, between $x$ and $y$ that is transverse to $\cF^\pm$. Equivalently there is an edge between $x$ and $y$ if and only if the pseudo-interval $\llbracket x,y \rrbracket$ is a true interval $[x,y]$.
\end{defi}

    It follows from the definition that $X^\pm$ is a subgraph of $\Gamma^\pm$, as one has an edge between $x,y$ in $X^\pm$ if and only if the transverse path between $x$ and $y$ in $P$ is realized by a segment of a leaf of $\cF^\mp$.

Unlike $X^\pm$, there is in general no quasi-isometry between $\Gamma^+$ and $\Gamma^-$.  For example, one can construct a bifoliated plane with no prong leaves and where the leaf space of $\cF^+$ is homeomorphic to $\mathbb{R}$, which implies that $\Gamma^+$ has diameter 1, while $\Gamma^-$ is unbounded.  For a concrete example, take the open region between the graphs $y=\sin(x)$ and $y=\sin(x)+1/2$ in $\mathbb{R}^2$ with the horizontal and vertical foliations.  
However, each is individually a quasi-tree: 

\begin{prop} 
$\Gamma^+$ and $\Gamma^-$ are quasi-trees.
\end{prop}

\begin{proof}
We show this for $\Gamma^+$, the case of $\Gamma^-$ being identical.  The proof follows very closely that of Proposition \ref{prop:quasi-tree}: first, for any path $\gamma$ formed by vertices $v_1, v_2,\ldots v_n$ in $\Gamma^+$ define the projection $\pi_{\Lambda^+}(\gamma)$ to be the $\cF^+$ saturation of the union (over $i$) of a sequence of transverse paths from $v_i$ to $v_{i+1}$ realizing the edges between adjacent vertices.  This does not depend on the choice of transverse path, because $\Lambda^+$ is simply connected, and by construction for any $x,y \in \Lambda^+$, the projection of any path from $x$ to $y$ in $\Gamma^+$ contains $\llbracket x, y \rrbracket$.  

The proof of Lemma \ref{lem:geodesics_vs_pseudo_interval} can be repeated verbatim to show that geodesic paths from $x$ to $y$ in $\Gamma^+$ project into the $2$-neighborhood of $\llbracket x, y \rrbracket$, so Manning's bottleneck criterion applies as before.  
\end{proof}

We also have: 

\begin{claim} \label{claim: wpd for non separated metric}
    Let $g \in G$ be an element satisfying Proposition \ref{prop: free element condition}, Item \ref{it: non sep}.
    Then it is a WPD element for the action of $G$ on $\Gamma^+$.
\end{claim}

\begin{proof}
Here the proof of Proposition \ref{prop:WPD} applies exactly as before:
The distance $d_{\Gamma^+}$ on the axis will be equal to the number of block intervals in the axis separating two points, therefore it will satisfy the estimate of Lemma \ref{lem: blocks are long}, and Lemma \ref{lem: projection minimize distance}. 
\end{proof}
We use this in the next section to show that, under very general conditions, loxodromic elements whose axis are a disjoint union of intervals, as in Item \ref{it: non sep}, are generic, see Theorem \ref{thm: genericity broken axis}. 

\subsection{Metrics on the plane and dualizable systems}\label{subsec: metric on plane} 

As we advertized in the introduction, our graph $X$ is a bifoliated plane version of Hagen's \emph{contact graph} for $\mathrm{CAT}(0)$ cube complexes \cite{hagen2014weak}. A slightly different point of view on Hagen's contact graph is to see it as a different choice of metric on the cube complex itself, where the distance between points is now measured by counting \emph{strongly separated} hyperplanes one can fit between two points (see \cite{Genevois20}).

In this section we will show how one can similarly interpret the distances on $X$, $X^\pm$ and $\Gamma^\pm$ as metrics on $P$, and compute distances by considering the maximal cardinality of sequences of leaves, with certain property, separating two points in $P$.

In fact, more then a parallel, both $\mathrm{CAT}(0)$ cube complexes and bifoliated planes can be seen as examples of a more general framework of \emph{dualizable systems}, as introduced by Petyt and Zalloum in \cite{Petyt-Z24}: This framework starts with a set $S$ together with a family of bipartitions $W$ of $S$, called \emph{walls}. Given this data, \cite{Petyt-Z24} introduces a \emph{dualizable system} $\mathcal C$ which is a set, satisfying certain properties, whose elements are sequences of walls. From this data, \cite{Petyt-Z24} builds a metric on $S$, which exhibits many features of a non-positively curved space (see \cite[Theorem K]{Petyt-Z24}), by taking the distance between two points in $S$ to be the maximal cardinality among the elements $c\in \mathcal C$ that separates the points.

This framework recovers a space quasi-isometric to the contact graph of Hagen by considering the dualizable system $\mathcal C$ to be any sequence of strongly separating hyperplanes (\cite{Petyt-Z24}). Here we will see that we can recover distances that are quasi-isometric to our graphs $X$, $X^\pm$ and $\Gamma^\pm$ by considering the set $S$ to be the plane $P$, walls to be leaves of $\cF^+$, $\cF^-$ or $\cF^+\cup\cF^-$ and appropriate choices of dualizable systems $\mathcal C$ that we introduce now:

\begin{defi}
A pair of disjoint leaves in $\cF^+\cup \cF^-$ is said to be \emph{hyperbolically aligned} if there is no third leaf intersecting both of them\footnote{This property is called strongly separated in the case of hyperplanes in a $\mathrm{CAT}(0)$ cube complex, but we chose a different name here to avoid confusion.}. A subset $\{l_1, \cdots l_k\}$ of $\cF^+ \cup \cF^-$ is said be hyperbolically aligned if each pair $l_i, l_j$ is. 
\end{defi}
We denote by $\mathcal{C}_H$, $\mathcal{C}_{H^+}$, and $\mathcal{C}_{H^-}$,  the collection of all hyperbolically aligned subsets of, respectively, $\cF^+ \cup \cF^-$, $\cF^+$ and $\cF^-$.

\begin{defi}
Say that a pair of disjoint leaves $l,l' \in \cF^+$ is  \emph{Reeb-separated} if the pseudo-interval $\llbracket l, l' \rrbracket$ is not an interval, i.e., if it contains a pair of non-separated leaves\footnote{Equivalently, $l,l'$ are Reeb-separated if one sees a generalized Reeb component ``in between'' $l$ and $l'$, hence our choice of name.}. A subset $c=\{l_1, \cdots l_k\}$ is said to be Reeb-separated if each pair $l_i,l_j$ is. 
\end{defi}
Paralleling the notation above, we denote the collection of all Reeb-separated subsets of $\cF^+$ by $\mathcal{C}_{R}^+$, and $\mathcal{C}_{R}^-$ similarly. 

Finally, just for this discussion, we say that a leaf $l$ \emph{separates} $x,y \in P$ if $x,y$ do not both lie on $l$ nor both lie in the same connected component of $P \setminus l$. Notice that this definition is nonstandard, as it allows for either $x$ or $y$ to be on $l$. We take this convention in order to ensure the triangle inequality in the metric $d_H$ we define below. 
A collection of leaves $c=\{l_1, \cdots l_k\}$ is said to \emph{separate} $x,y$ if each $l_i$ does.

Using the terminology above, we make the following definitions
\begin{defi}
    For any $x,y\in P$ we set:
    \begin{enumerate}
    \item $d_H(x,y):=\sup\{|c|: c \in \mathcal{C}_H \text{ and } c \text{ separates  }x,y\}$;
    \item $d_\pm(x,y):=1+\sup\{|c|: c \in \mathcal{C}_{H^\pm} \text{ and } c \text{ separates  }x,y\}$ if $x\neq y$ and $d_\pm(x,x)=0$;
     \item $d_{R^\pm}(x,y):=1+ \sup\{|c|: c \in \mathcal{C}_{R^\pm} \text{ and } c \text{ separates  }x,y\}$ if $x\neq y$ and $d_{R^\pm}(x,x)=0$.
\end{enumerate}
\end{defi}

We will next show that the functions defined above are actually metrics on $P$. Notice that the metric $d_{R^\pm}$ has actually been defined and used, not on the plane $P$ but on the leaf space $\Lambda^\pm$, already in the work of Barbot \cite{Bar98}.

\begin{prop}
    The five functions $d_H, d_\pm, d_{R^\pm}$ are all metrics on $P$.
\end{prop}
\begin{proof}
    By definition, $d_{R^\pm}(x,y)=0$ and $d_\pm(x,y)$ if and only if $x=y$. Since any two distinct points in $P$ are separated by at least one leaf of $\cF^+\cup \cF^-$, we also immediately get that $d_H(x,y)=0$ if and only if $x=y$.

Before proving the triangle inequality, we show that the functions are all finite: Given distinct points $x,y$, choose a path $\gamma$ in $P$ between $x$ and $y$. We further choose $\gamma$ so that it does not contain any prong singularities, except possibly $x$ and $y$. One can then cover $\gamma$ by finitely many trivially foliated rectangles. In particular, for any $c=\{l_1, \cdots l_k\}$ in $\cC_H$, $\cC_{H^\pm}$, or $\cC_{R^\pm}$ that separates $x$ and $y$, no two distinct $l_i,l_j$ can hit the same trivially foliated rectangle. So the number of elements in such a collection is bounded.

    Finally, to prove the triangle inequality, it is enough to notice that if $c=\{l_1, \cdots l_k\}$ separates $x,y$ and $c'=\{l_1', \cdots l_n'\}$ separates $y,z$ then any leaf $l_i$ either separates $x,z$ or separates $y,z$. In particular, $c$ splits into a union $c=c_1\cup c_2$ where $c_1$ separates $x,z$ and $c_2$ separates $y,z$. Similarly, $c' = c_1' \cup c_2'$ with $c_1'$ separating $x,y$ and $c_2'$ separating $x,z$.
    So, irrespectively of whether $c$ and $c'$ are in $\cC_H$, $\cC_{H^\pm}$, or $\cC_{R^\pm}$ we always have that there will be less element in a collection separating $x,z$ than the sum of the elements in a maximal collection separating $x,y$ and a maximal collection separating $y,z$.    
\end{proof}

\begin{prop}
    There exist quasi-isometries $f_H\colon (P, d_H) \to X$, $f_\pm\colon (P, d_\pm) \to X^\pm$ and $f_{R^\pm}\colon (P, d_{R^\pm}) \to \Gamma^\pm$.
\end{prop}

The strategy for the proof of this result is essentially the same for all the different spaces. Since we have been working with $X^+$ for most of this article, we will do the proof in that case and leave as an exercise the adaptation to the other spaces.

\begin{proof}
    Define a map $f_+ \colon (P, d_+) \to X^+$ by associating to a point $x\in P$ the vertex in $X^+$ corresponding to the leaf $\cF^+(x)$.

    Let $x,y \in P$ and let $\gamma$ be a geodesic in $X^+$ connecting $v,w\in X^+$ where $f_+(x) = v, \, f_+(y)= w.$ Let $\gamma_P$ be a realization of $\gamma$ in $P$ (see Definition \ref{def_projection_maps}).
    Fix a collection of hyperbolically aligned leaves $\{h_1,h_2 \cdots h_n\}$ in $\cF^+$ separating $x,y$ and realizing the distance $d_+(x,y)$, i.e., such that $d_+(x,y)= n+1$. Since each $h_i$ disconnects $P$ into two path connected components, $\gamma_P$ must intersect each $h_i$. As in the proof of Lemma \ref{lem: blocks are long}, since no leaf of $\cF^-$ can intersect both $h_i$ and $h_{i+1}$, it implies that $\gamma$ must contain at least one vertex in between $h_i$ and $h_{i+1}$.
In particular, this implies that $|\gamma|=d_{X^+}(v,w) \geq n-1 = d_+(x,y) - 2$.

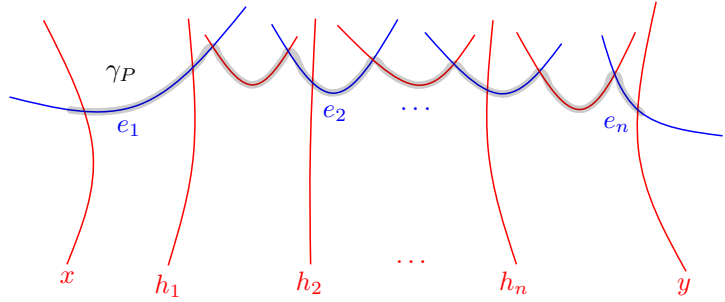
\begin{figure}[h]
    \centering
\labellist
\small

\pinlabel{\color{red} $x$} [t] at 67 34
\pinlabel{\color{red} $h_1$} [t] at 134 32
\pinlabel{\color{red} $h_2$} [t] at 228 34
\pinlabel{\color{red} $\dots$} [l] at 280 35
\pinlabel{\color{red} $h_n$} [t] at 364 34
\pinlabel{\color{red} $y$} [t] at 477 29

\pinlabel{\color{blue} $e_1$} [t] at 108 134
\pinlabel{\color{blue} $e_2$} [t] at 245 144
\pinlabel{\color{blue} $\dots$} [t] at 300 143
\pinlabel{\color{blue} $e_n$} [tr] at 445 135

\pinlabel{$\gamma_P$} at 103 160

\endlabellist

    \includegraphics[width=0.8\linewidth]{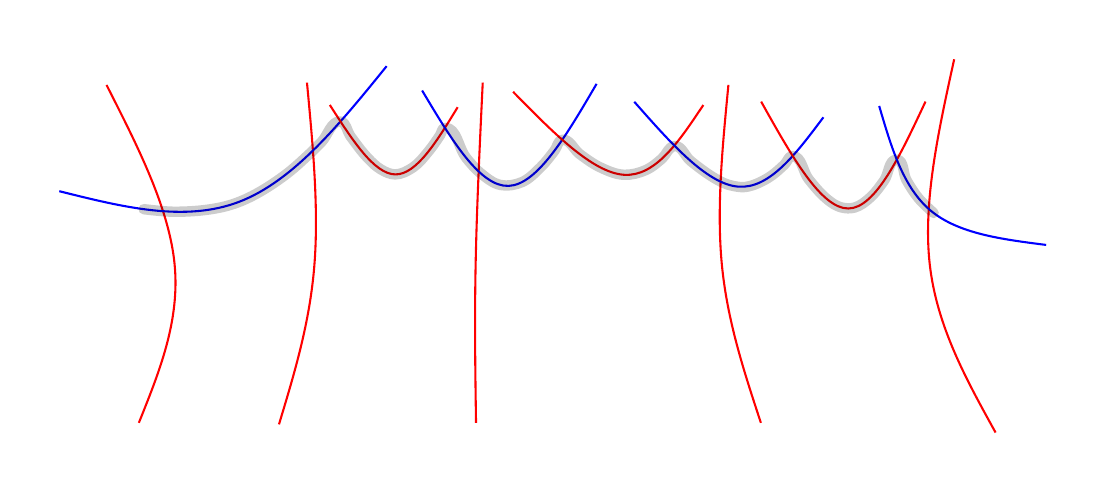}
    \caption{The polygonal path $\gamma_P$ intersecting the set of hyperbolically aligned leaves.}
    \label{fig: dH and dX}
\end{figure}

Next we claim $d_{X^+}(v,w) \leq 5 d_+(x,y)$.  To prove this, suppose for contradiction there exists $i$ such that $\gamma$ has $5$ vertices, $v_1,\dots, v_5$ between $h_i$ and $h_{i+1}$, then no leaf of $\cF^-$ can intersect both $v_1$ and $v_3$ nor can there exists a leaf of $\cF^-$ intersecting both $v_3$ and $v_5$. Thus, $\{h_1, \dots, h_i, v_3, h_{i+1}, \dots h_n\}$ is an hyperbolically aligned set that separates $x$ and $y$, contradicting that $n= d_+(x,y)$.

\begin{figure}[h]
    \centering
\labellist
\small

\pinlabel{\color{red} $x$} [t] at 28 89
\pinlabel{\color{red} $h_i$} [t] at 95 65
\pinlabel{\color{red} $h_{i+1}$} [t] at 318 49
\pinlabel{\color{red} $y$} [t] at 385 85

\pinlabel{\color{red} $v_1$} [b] at 131 248
\pinlabel{\color{red} $v_2$} [b] at 163 271
\pinlabel{\color{red} $v_3$} [b] at 205 256
\pinlabel{\color{red} $v_4$} [b] at 245 274
\pinlabel{\color{red} $v_5$} [b] at 294 269

\pinlabel{\color{blue} $e_1$} [t] at 145 231
\pinlabel{\color{blue} $e_2$} [b] at 185 79
\pinlabel{\color{blue} $e_3$} [t] at 221 211
\pinlabel{\color{blue} $e_4$} [t] at 278 77

\endlabellist

    \includegraphics[width=0.8\linewidth]{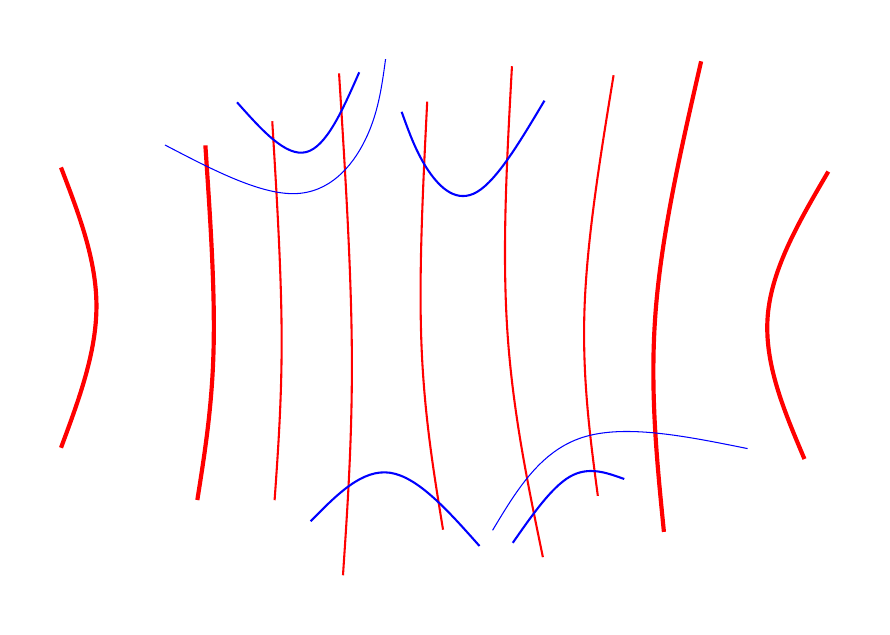}
    \caption{If 5 vertices $v_1, \dots, v_5$ are between $h_i, h_{i+1}$ then $\{h_i, v_3, h_{i+1} \}$ are hyperbolically aligned}
    \label{fig: five vertices}
\end{figure}
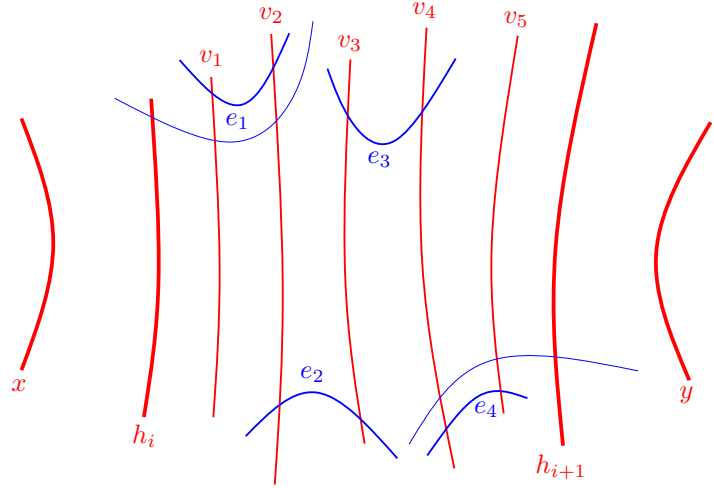

By a similar argument, $\gamma$ cannot have $3$ or more vertices between $x$ and $h_1$ or $h_n$ and $y$.
Hence, we deduce that $|\gamma|=d_{X^+}(v,w) \leq 5 (n + 1) = 5 d_+(x,y)$. Therefore, we showed that for all $x,y\in P$, 
\[
 d_+(x,y) - 2 \leq d_{X^+}(f_+(x), f_+(y)) \leq 5 d_+(x,y).
\]
As $f_+$ is surjective, we deduce that it is a quasi-isometry.

The maps $f_-$ and $f_{R^\pm}$ can be defined analogously as $f_+$ and the proof that they are quasi-isometries is almost identical. For the map $f_H\colon (P,d_H) \to X$, one can define it by taking $f_H(x)$ to be the $\cF^+$ leaf through $x$ (or the $\cF^-$-leaf), which is now coarsely surjective, and the proof is similar.
\end{proof}

\begin{rmk} Notice that the metrics $d_H, d_\pm, d_{R^\pm}$ are all \emph{roughly geodesic} metrics, i.e., there exists a constant $C$ such that every pair of points $x,y \in P$ are connected by a $(1,C)$-quasi-geodesic in each of the metrics $d_H, d_\pm, d_{R^\pm}$. To construct such a quasi-geodesic, it is enough to note that any path between $x$ and $y$ can be isotoped to be in a "maximally transverse" position in the sense that it is tangent to either foliations in at most finitely many points and that it intersect any leaf at most twice. One can show that such a path will automatically be a $(1,C)$-quasi-geodesic.

Rough geodesicity is an important property, as it is this class of metric spaces in which hyperbolicity is a quasi-isometry invariant.  
\end{rmk}

\section{Genericity of free elements and critical exponent} \label{sec: genericity}

In this section we deduce Theorem 
\ref{thm_generic_flows_version}, as well as Corollaries \ref{cor: random walks} and \ref{cor: typical geodesic}.

We in fact prove a more general version of Theorem \ref{thm_generic_flows_version} about Anosov-like actions:
\begin{thm} \label{thm_general_genericity} 
Let $G$ be a finitely generated group acting Anosov-like on $P$ with an extremal Smale class. Assume that $P$ is neither skew nor trivial.
Let $\cP(G)$ be the elements of $G$ admitting a fixed point in $P$ and $\cP(G)^c$ its complement. Then, for any word metric on $G$, the set $\cP(G)^c$ is generic. 

Furthermore, for any finite generating set $S$, there exists $S' \supset S$ for which the critical exponent of $\cP(G)^{c}$ is strictly larger than that of $\cP(G)$. If $G$ is hyperbolic, this holds with $S'=S$.
\end{thm}

\begin{proof}
By Theorem \ref{thm_generic_plane_version}, $G$ acts on a hyperbolic space $X$ with a loxodromic WPD element.
Since $G$ is assumed to be finitely generated, by \cite[Theorem 1.1]{Choi2025acylindrically}, we have for any generating set $S$,

$$\frac{|\{g \in G: g \text{ is a WPD loxodromic on }X\} \cap B_S(e,n) |}{|B_S(e,n)|}  \rightarrow 1.$$

In particular, as every loxodromic element is in $\cP(G)^c$, the set $\cP(G)^c$ is generic for any word metric.

The statement about critical exponents follows from \cite[Theorem A]{Choi2024PseudoAnosovsGeneric}, which states that the critical exponent of a subset with translation length bounded (in particular containing $\cP(G)$), is strictly smaller than the critical exponent of the group.

Now as hyperbolic groups admit a thick bicombing for every generating set \cite[Lemma 8.1]{GekhtmanTaylorTiozzo2022}, any hyperbolic group $G$ satisfies the assumption of \cite[Theorem 1.1-(2)]{GekhtmanTaylorTiozzo2022} for every generating set, which implies again that the critical exponent of $\cP(G)$ is strictly smaller than the critical exponent of the group.
\end{proof}

Finally, we prove corollaries \ref{cor: random walks} and \ref{cor: typical geodesic}.
\begin{proof}[Proof of Corollary \ref{cor: random walks}]
 As $G$ acts Anosov-like on a nonskew, nontrivial, bifoliated plane, $G$ must be countable \cite[Corollary 3.11]{Cam25}. Therefore, considering the action of $G$ on the hyperbolic graph $X^+$, \cite[Theorem 1.3]{MT18} implies that for every vertex $v\in X^+$ and almost every sample path $(g_n)$ in $G$, $(g_nv)$ converges to a point of the Gromov-boundary $\partial X^+$. This implies that the sequence of leaves $(g_nv)$, considered as a sequence of subsets in $P$, must escape every compact set in $P$, since the $\cF^+$-saturation of any compact set in $P$ projects to a bounded set in $X^+$.
As there are no infinite product region, we deduce that $(g_nv)$ converges to a point $\xi\in\partial P$, and in particular, for every $x_0\in v$, $g_nx_0\to \xi$, proving the corollary.
\end{proof}

\begin{proof}[Proof of Corollary \ref{cor: typical geodesic}]
Let $\nu$ be Patterson--Sullivan measure for the given generating set. Since $G$ is assumed to be hyperbolic, applying \cite[Theorem 1.2]{GTT18} to the action of $G$ on $X^+$ gives that for any $v\in X^+$, for $\nu$-almost every point $\eta\in \partial G$ and any geodesic ray $(g_n)$ converging to $\eta$, the sequence $(g_nv)$ converges to a point in $\partial X^+$. As in the previous proof, this implies that for any $x_0\in v$, $(g_nx_0)$ converges to a single point in $\partial P$.
\end{proof}

\subsection{Genericity of broken axes} Recall that the axis of any element is either homeomorphic to $\mathbb{R}$ (which can be properly embedded or not in the leaf space) or of the form $\cA(g) = \bigcup_i [x_i,y_i]$ with $y_i$ non-separated from $x_{i+1}$; in this latter case we call it a {\em broken axis}.  We next show that, under general assumptions, generic elements not only act freely but have broken axes: 

\begin{thm} \label{thm: genericity broken axis}
    Let $G \subset \Aut(P)$ be Anosov-like with an extremal Smale class, with $P$ neither trivial nor skew. Assume that either $P$ has no prongs, or that $P$ admits at least one pair of non-separated leaves and $G$ acts topologically transitively on $P$.
    Then the the set of free elements of $\cP(\phi)^c$ with broken axis is generic.
\end{thm}

\begin{rmk}
    In fact, elements with broken axis are generic as soon as $P$ admits at least one pair of non-separated leaves.  However, the proof of this stronger fact requires more work and the import of many more results from \cite{barthelmeNONTRANSITIVEPSEUDOANOSOVFLOWS}, which is why we only explain the transitive case here.
\end{rmk}

\begin{proof}
Let $\Gamma^+$ be the graph constructed in Definition \ref{def_broken_axis_graph}. Note that by definition, \emph{every} loxodromic element for the action on $\Gamma^+$ must have a broken axis.

By Claim \ref{claim: wpd for non separated metric}, any element $g\in G$ that satisfies item \ref{it: non sep} of Proposition \ref{prop: free element condition} is a WPD element on $\Gamma^+$. So, as soon as we can prove the existence of one such element, \cite[Theorem 1.1]{Choi2025acylindrically} gives genericity of those elements for any finite generating set.

If $P$ has no prong singularities, then every element constructed in Proposition \ref{prop: free element condition} has a broken axis, so there is nothing more to prove in that case.

Now assume that $P$ admits at least one pair of non-separated leaves $l_1,l_2$. Since $G$ acts topologically transitively on $P$, the set of points fixed by non-trivial elements of $G$ is dense in $P$ (see \cite[Theorem 2.9.2]{barthelmePseudoAnosovFlowsPlane2026}), and as $P$ is neither trivial nor skew, the set of \emph{non-corner} fixed points is also dense (see \cite[Lemma 2.30]{barthelmeOrbitEquivalencesPseudoAnosov2022}). Using this, it is immediate that one can find non-corner fixed points $a,b$ in the same configuration as in the left-hand side of Figure \ref{fig: broken and prongs}, i.e., such that exactly one of the two non-separated leaves $l_1,l_2$ separates $a$ from $b$. Therefore, the proof of Proposition \ref{prop: free element condition} give an element with a broken axis for this choice of $a$ and $b$, and the result follows.
\end{proof}

\subsection{Special cases: trivial and skew planes}\label{subsec: skew and trivial}

We end this article by discussing the two cases from the trichotomy, namely trivial and
skew planes, where the action on the hyperbolic graph $X^+$ fails to be non-elementary. In particular, we will show that the conclusion of Theorem~\ref{thm_generic_plane_version}
cannot be extended to the $\R$-covered setting.

\subsubsection*{The trivial plane}

We first treat the case of a trivial plane. Recall that in the flow setting this corresponds to the flow being a suspension of an Anosov diffeomorphism of the torus. The situation is very different in this setting from the conclusion of Theorem \ref{thm_generic_plane_version}: 

\begin{prop}
\label{prop: trivial case}
Let $G$ be a finitely generated group acting Anosov-like on the trivial plane $P_{\mathrm{triv}}$.
Then the set of elements acting with a fixed point is generic in $G$ and the set of free elements has polynomial growth.
\end{prop}

\begin{proof}
By \cite[Proposition 2.3.2]{barthelmePseudoAnosovFlowsPlane2026}, if $G$ acts Anosov-like on the trivial plane, then $G$ is conjugate to a subgroup of affine transformations, and thus splits as a sequence $A \to G \to H$ where $H$ is the linear part of $G$ under the derivative map to $\mathrm{GL}(n, \mathbb{R})$ and the kernel $A$ is the translational part.  The set of elements acting freely is precisely $A$, which is abelian and hence has polynomial growth.  By contrast, the Anosov-like axiom \ref{Axiom_dense} implies that $G$ contains hyperbolic linear elements with distinct $x$ and $y$ coordinates for their fixed points; it is easy to verify that (up to switching one such element with its inverse), these generate a free semi-group, hence $G$ itself has exponential growth. 
\end{proof}

\subsubsection*{The skew plane}

In the case of an Anosov-like action on the skew plane, while we cannot prove genericity of elements acting freely, we prove that they are still a large subset: 
\begin{thm}
\label{thm: skew free maximal exponent}
Let $G$ be a finitely generated group acting Anosov-like on the skew plane $P_{\mathrm{skew}}$.
Then for any finite generating set $S$ on $G$, the set $\mathrm{Free}$ of elements acting freely
has maximal critical exponent:
\[
\lambda_S(\mathrm{Free})=\lambda_S(G),
\]
Moreover,
\[
\limsup_{n\to\infty}\frac{|\mathrm{Free}\cap B_S(n)|}{|B_S(n)|}>0.
\]
\end{thm}

In order to prove the theorem, we will use the two following lemmas about growth in groups.

\begin{lem}
\label{lem: ball comparison}
Let $G$ be a finitely generated group, $S$ a finite generating set, and let $k \geq 0$.
Then there exists a constant $L \geq 1$ such that for all $n \in \mathbb N$,
\[
|B_S(n+k)| \leq L\, |B_S(n)|.
\]
\end{lem}

\begin{proof}
We iterate the estimate
$|B_S(n+1)| \leq 2|S|\,|B_S(n)|$ and take $L=(2|S|)^k$.
\end{proof}

\begin{lem}
\label{lem: finite translate same exponent}
Let $G$ be a finitely generated group, $S$ a finite generating set, and let $A\subset G$.
Assume that there exists a finite set $F\subset G$ such that for every $g\in G$,
$gF\cap A \neq \emptyset.$
Then
$\lambda_S(A)=\lambda_S(G).$
\end{lem}

\begin{proof} Let $k:= \max \{ |f|_S: f \in F \}$ and let $R := | B_S(k) |$ be the size of a ball of radius $k.$ 
For every $g\in B_S(n)$ there exists $f\in F$ such that $gf\in A$, and we have
$|gf|_S\le n+k$. Hence
$|B_S(n)| \le R\,|A\cap B_S(n+k)|$.
By Lemma~\ref{lem: ball comparison}, there exists $L\ge1$ such that
$|B_S(n+k)| \le L\,|B_S(n)|.$
Hence, combining the two equations we get $$|B_S(n+k)| \leq LR|B_S(n+k) \cap A|.$$ Taking the log of both sides and dividing by $n+k$ gives 

$$\frac{\ln|B_S(n+k)|}{n+k} \leq \frac{\ln|LR|}{n+k}+\frac{\ln |A \cap B_S(n+k)|}{n+k}.$$ 

Passing to the limsup on both sides yields
$\lambda_S(G)\le \lambda_S(A)$.
The reverse inequality is automatic, so equality holds.
\end{proof}

\begin{proof}[Proof of Theorem \ref{thm: skew free maximal exponent}]
An Anosov-like action on the skew plane gives an identification of $G$ with a subgroup of $\mathrm{Homeo}_{\mathbb Z}(\mathbb R)$, the group of homeomorphisms of $\R$ which commute or anti-commute with integer translations (see \cite[Section 2.3]{barthelmePseudoAnosovFlowsPlane2026}).

Choose an element $h\in G$ acting freely on $\R$; such an element exists by
\cite[Lemma 2.2]{barthelmeOrbitEquivalencesMathbb2023}. Replacing $h$ by a power, we may assume that
$h(x)>x+1$ for all $x \in \R$.
Now let $g\in G$. Either $g$ acts freely, or
$g$ has a fixed point $x\in\mathbb R$ and then
$hg(x)=h(x)>x+1$,
so $hg$ acts freely.

We have shown that for every $g\in G$, at least one of $g$ or $hg$ belongs to
$\mathrm{Free}$. Equivalently,
\[
g\{1,h\}\cap \mathrm{Free}\neq\emptyset \qquad \text{for all } g\in G.
\]
Applying Lemma~\ref{lem: finite translate same exponent} with $F=\{1,h\}$ gives
$\lambda_S(\mathrm{Free})=\lambda_S(G)$ for any generating set $S$.

To obtain the second part, let $R=|h|_S$.
Denote $K=|B_S(R)|$. For every $g\in B_S(n)$, either $g\in \mathrm{Free}$ or $hg\in \mathrm{Free}\cap B_S(n+R)$.
Hence
\[
|B_S(n)| \le K \, . \, | \mathrm{Free}\cap B_S(n+R) |
\]
Using Lemma~\ref{lem: ball comparison}, there exists $L\ge 1$ such that
$|B_S(n+R)|\le L\,|B_S(n)|$ for each $n \in \N$.
Therefore
\[
\frac{|\,\mathrm{Free}\cap B_S(n+R)\,|}{|B_S(n+R)|}
\ge \frac{1}{LK},
\]
and taking the limsup gives the conclusion.
\end{proof}

\begin{rmk}
Unlike in the non-$\R$-covered case, we do \emph{not} prove that the
set of elements acting freely is generic, nor that the set of elements with fixed points
has smaller critical exponent. In fact, in the case of a geodesic flow on a hyperbolic surface, if we consider the Sasaki metric on $T^1S$, then one can show that the critical exponent of elements representing periodic orbits (counted with respect to the geodesic length for the Sasaki metric) is \emph{equal} to the critical exponent of the set of elements acting freely. We did not however investigate whether this also holds for some or all word metrics.
\end{rmk}

\bibliographystyle{alpha}
\bibliography{biblio}

\end{document}